\documentclass[12pt, twoside]{article}
\usepackage{amsfonts,mathrsfs,amsmath,amssymb,amsthm,upref,amscd}
\usepackage{xcolor}
\usepackage[all,cmtip]{xy}
\usepackage{setspace}
\usepackage{enumerate}
\usepackage[titletoc]{appendix}
\usepackage{times}
\usepackage{cite}
\usepackage{tikz}
\usepackage[colorinlistoftodos]{todonotes}
\usetikzlibrary{arrows}
\numberwithin{equation}{section}

\makeatletter 
\@mparswitchfalse
\makeatother
\normalmarginpar 

\def\titlerunning#1{\gdef\titrun{#1}}
\makeatletter
\def\author#1{\gdef\autrun{\def\and{\unskip, }#1}\gdef\@author{#1}}
\def\address#1{{\def\and{\\\hspace*{18pt}}\renewcommand{\thefootnote}{}%
		\footnote {#1}}%
	\markboth{\autrun}{\titrun}}
\makeatother
\def\email#1{e-mail: #1}
\def\subjclass#1{{\renewcommand{\thefootnote}{}%
		\footnote{\emph{Mathematics Subject Classification (2010):} #1}}}
\def\keywords#1{\par\medskip
	\noindent\textbf{Keywords.} #1}

\allowdisplaybreaks

\theoremstyle{plain}
\newtheorem{Thm}{Theorem}[section]
\newtheorem{Lem}[Thm]{Lemma}
\newtheorem{Cor}[Thm]{Corollary}
\newtheorem{Prop}[Thm]{Proposition}
\newtheorem*{Thm*}{Theorem}
\newtheorem*{claim*}{Claim}
\newtheorem*{Cor*}{Corollary}
\newtheorem{Ques}{Question}
\newtheorem*{Ques*}{Question}

\newtheorem*{Prob*}{Problem}
\newtheorem*{OProb*}{Open Problem}
\theoremstyle{definition}
\newtheorem{Def}[Thm]{Definition}
\newtheorem*{Def*}{Definition}
\newtheorem{Rem}[Thm]{Remark}

\DeclareMathOperator{\tr}{tr}
\DeclareMathOperator{\real}{Re}
\DeclareMathOperator{\cat}{Cat}

\DeclareMathOperator{\vol}{vol}

\DeclareMathOperator{\supp}{supp}

\DeclareMathOperator{\tdeg}{deg}

\DeclareMathOperator{\diag}{diag}
\DeclareMathOperator{\Vol}{Vol}
\DeclareMathOperator{\Spin}{Spin}
\DeclareMathOperator{\SO}{SO}
\DeclareMathOperator{\Aut}{Aut}
\DeclareMathOperator{\diam}{diam}

\newcommand{\equ}{equation}
\newcommand{\C}{\mathbb{C}}
\newcommand{\N}{\mathbb{N}}
\newcommand{\R}{\mathbb{R}}

\newcommand\eps{\varepsilon}

\newcommand\ce{\mathcal{E}}

\newcommand\cg{\mathcal{G}}
\newcommand\ch{\mathcal{H}}
\newcommand\ci{\mathcal{I}}
\newcommand\cj{\mathcal{J}}

\newcommand\cm{\mathcal{M}}
\newcommand\cn{\mathcal{N}}

\newcommand\cw{\mathcal{W}}

\newcommand{\inp}[2]{\left\langle#1,#2\right\rangle}
\newcommand{\normm}[1]{{\left\vert\kern-0.25ex\left\vert\kern-0.25ex\left\vert #1 
		\right\vert\kern-0.25ex\right\vert\kern-0.25ex\right\vert}}

\def\mbs{\mathbb{S}}

\def\msd{\mathscr{D}}

\def\msf{\mathscr{F}}

\def\msh{\mathscr{H}}

\def\mfm{\mathfrak{M}}
\def\id{\text{Id}}
\def\ig{\textit{g}}
\def\ih{\textit{h}}

\def\ov{\overline}

\def\pa {\partial}

\def\De{\Delta}

\def\al{\alpha}
\def\bt{\beta}

\def\de{\delta}
\def\Ga{\Gamma}
\def\ga{\gamma}

\def\lm{\lambda}

\def\om{\omega}
\def\Om{\Omega}
\def\sa{\sigma}
\def\vr{\varepsilon}
\def\va{\varphi}

\begin{document}
	
	\titlerunning{Non-compactness results for the spinorial Yamabe-type problems}
	
	\title{Non-compactness results for the spinorial Yamabe-type problems with non-smooth geometric data}
	
	\author{Takeshi Isobe \quad Yannick Sire \quad  Tian Xu}
	
	\date{}
	
	\maketitle
	
	\address{T. Isobe: Graduate School of Economics, Hitotsubashi University, 2-1 Naka, Kunitachi, Tokyo 186-8601, Japan;
		\email{t.isobe@r.hit-u.ac.jp}
		\and
		Y. Sire: Department of Mathematics, Johns Hopkins University, 3400 N. Charles Street, Baltimore, Maryland 21218;
		Y.S. is partially supported by NSF DMS Grant $2154219$, " Regularity {\sl vs} singularity formation in elliptic and parabolic equations". \email{ysire1@jhu.edu}. 
		\and
		T. Xu: Department of Mathematics, Zhejiang Normal University, Jinhua, Zhejiang, 321004, China;
		T.X. is partially supported by the National Science Foundation of China (NSFC 11601370) and the Alexander von Humboldt Foundation of Germany.
		 \email{xutian@amss.ac.cn} 
	}
	
	
	\begin{abstract}
		
	Let $(M,\ig,\sa)$ be an $m$-dimensional closed spin manifold, with a fixed Riemannian metric $\ig$ and a fixed spin structure $\sa$; let $\mbs(M)$ be the spinor bundle over $M$. The spinorial Yamabe-type problems address the solvability of the following equation
	\[
	D_\ig\psi=f(x)|\psi|_\ig^{\frac2{m-1}}\psi, \quad \psi:M\to\mbs(M), \ x\in M
	\]
	where $D_\ig$ is the associated Dirac operator and $f:M\to\R$ is a given function. The study of such nonlinear equation is motivated by its important applications in Spin Geometry: when $m=2$, a solution corresponds to a conformal isometric immersion of the universal covering $\widetilde M$ into  $\R^3$ with prescribed mean curvature $f$; meanwhile, for general dimensions and $f\equiv constant\neq0$, a solution provides an upper bound estimate for the B\"ar-Hijazi-Lott invariant.
		
	The aim of this paper is to establish non-compactness results related to the spinorial Yamabe-type problems. Precisely, concrete analysis is made for two specific models on the manifold $(S^m,\ig)$ where the solution set of the spinorial Yamabe-type problem is not compact: $1).$ the geometric potential $f$ is constant (say $f\equiv1$) with the background metric $\ig$ being a $C^k$ perturbation of the canonical round metric $\ig_{S^m}$, which is not conformally flat somewhere on $S^m$; $2).$ $f$ is a perturbation from constant and is of class $C^2$, while the background metric $\ig\equiv\ig_{S^m}$.

		\vspace{.5cm}
	\end{abstract}
	
	\tableofcontents
	
	\section{Introduction}

    On a closed Riemannian $m$-manifold $(M, \ig)$ with $m\geq3$, the scalar curvature problem (or simply known as the Yamabe-type problem) is given by the differential equation
    \begin{\equ}\label{Yamabe}
    	-\frac{4(m-1)}{m-2}\De_{\ig}u + R_\ig u = f(x) u^{\frac{m+2}{m-2}}, \quad u>0,
    \end{\equ}
    where $\De_\ig$ is the Laplace operator with respect to $\ig$ and $R_\ig$ stands for the scalar curvature of $\ig$. Here, the problem is to decide which function $f$ on $M$ can be the scalar curvature of a conformal metric $\tilde\ig=u^{4/(m-2)}\ig\in[\ig]$. In case $f\equiv constant$, this problem is referred to as the classical Yamabe problem, and is completely solved by a series of works of Yamabe \cite{Yamabe}, Trudinger \cite{Trudinger}, Aubin \cite{Aubin} and Schoen \cite{Schoen}. See also the survey paper \cite{LeeParker} by Lee \& Parker.
    
    In the setting of spin geometry there exists a conformally covariant operator, the Dirac operator, which enjoys analogous properties to the conformal Laplacian. This operator was formally introduced by M.F. Atiyah in 1962 in connection with his elaboration of the index theory of elliptic operators.

    Let $(M,\ig,\sa)$ be an $m$-dimensional closed spin manifold, $m\geq2$, with a fixed Riemannian metric $\ig$ and a fixed spin structure $\sa:P_{\Spin}(M)\to P_{\SO}(M)$. The Dirac operator $D_\ig$ is defined in terms of a representation $\rho:\Spin(m)\to\Aut(\mbs_m)$ of the spin group which is compatible with Clifford multiplication. Let $\mbs(M):=P_{\Spin}(M)\times_\rho\mbs_m$ be the associated bundle, which we call the spinor bundle over $M$, with $\dim_\C\mbs(M)=2^{[\frac m2]}$. Then the Dirac operator $D_\ig$ is a first order differential operator acting on smooth sections of $\mbs(M)$, i.e. $D_\ig: C^\infty(M,\mbs(M))\to C^\infty(M,\mbs(M))$. We are concerned with the spinorial Yamabe-type problem
    \begin{\equ}\label{SY}
    D_\ig\psi=f(x)|\psi|_{\ig}^{\frac2{m-1}}\psi, \quad \psi: M\to\mbs(M)
    \end{\equ}
where $|\cdot|_\ig$ is the hermitian metric on $\mbs(M)$ induced from $\ig$. This equation appears in the study of different problems from conformal geometry, and has attracted much attention recently, see for instance \cite{Ammann, Ammann2003, Ammann2009, AGHM, AHM, Ba-Xu-JFA21, BF20, BMW21, CJW, Isobe11, Isboe-Xu21, Raulot, SX2020, SX2021} and references therein. 

We point out here that there are at least two motivations for studying Eq. \eqref{SY}.
One of them is that, when $f\equiv constant\neq0$ (say $f\equiv 1$), Eq. \eqref{SY} is closely related to the study of a conformal spectral invariant, i.e., the {\it B\"ar-Hijazi-Lott} invariant (see \cite[Section 8.5]{Ginoux} for an overview)
    \[
    \lm_{min}^+(M,\ig,\sa):=\inf_{\tilde\ig\in[\ig]}\lm_1^+(\tilde\ig)\Vol(M,\tilde\ig)^{\frac1m},
    \]
where $\lm_1^+(\tilde\ig)$ stands for the smallest (i.e. first) positive eigenvalue of $D_{\tilde\ig}$ with respect to $\tilde\ig\in[\ig]$. In fact, as was pointed out in \cite{Ammann, Ammann2009, AGHM}, the value of the B\"ar-Hijazi-Lott invariant for an arbitrary closed spin $m$-manifold can not be larger than that for the round sphere (with the same dimension), that is
    \begin{\equ}\label{BHL-inequ}
    \lm_{min}^+(M,\ig,\sa)\leq\lm_{min}^+(S^m,\ig_{S^m},\sa_{S^m})=\frac m2 \om_m^{\frac1m}
\end{\equ}
where $\ig_{S^m}$ is the standard round metric, $\sa_{S^m}$ stands for the unique spin structure on $S^m$ and $\om_m$ denotes the volume of $(S^m,\ig_{S^m})$. In this regard, the next stage would consist in showing that \eqref{BHL-inequ} is a strict inequality when $(M,\ig)$ is not conformally equivalent to $(S^m,\ig_{S^m})$. And it is important to notice that, if there exists a nontrivial solution to Eq. \eqref{SY} (with $f\equiv 1$) such that $\int_M|\psi|_\ig^{2m/(m-1)}d\vol_{\ig}<(\frac m2)^m\om_m$, then the strict inequality in \eqref{BHL-inequ} holds true (see \cite{SX2020, SX2021}). This can be viewed as the spinorial analogue of the Yamabe problem in geometric analysis. However, the strict inequality in \eqref{BHL-inequ} is only verified for some special cases (for instance, if $M$ is locally conformally flat, if $D_\ig$ is invertible and if the so-called {\it Mass endomorphism} is not identically zero \cite{AHM}, and all rectangular tori \cite{SX2020}, and non-locally conformally flat manifolds \cite{SX2021}), but a general result is still lacking (cf. \cite{ADHH, Ginoux, Hermann10}). The methods that can be used are sometimes similar to the ones of the Yamabe problem, but since we work with Dirac operator and spinors, the reasoning is more involved as the eigenvalues of the Dirac operator tend to both $-\infty$ and $+\infty$ and there is no adequate replacement for the maximum principle. 

Another reason that makes Eq. \eqref{SY} interesting is that, in dimension $m=2$, its solution provides a strong tool for showing the existence of prescribed mean curvature surfaces in $\R^3$ (here the function $f$ plays the role of the mean curvature). Special cases of such surfaces are constant mean curvature (CMC) surfaces (that is $f\equiv constant$) which have been studied before by completely different techniques, see for instance \cite{GB98}. The correspondence between a solution of Eq. \eqref{SY} on a Riemannian surface $M$ and a periodic conformal immersion (possibly with branching points) of the universal covering $\widetilde M$ into $\R^3$ with mean curvature $f$ is known as the spinorial Weierstra\ss\ representation. For details in this direction, we refer to \cite{Ammann, Ammann2009, Kenmotsu, Konopelchenko, KS96, Fredrich98, Matsutani, Taimanov97, Taimanov98, Taimanov99} and references therein.

Although the existence problem for Eq. \eqref{SY} is not settled in full generality, there are several partial existence results in the literature, see for instance \cite{Isobe13, Isobe15, Xu-CAG}, and it is often true that many solutions exist for Eq. \eqref{SY}. As a first step towards multiple existence results, consider the problem on the Torus $S^1(L)\times S^1(1)$ with product metric, there are many non-minimizing solutions if $L$ is large, see \cite{SX2020} (and also see \cite{Isboe-Xu21} for more examples in the non-locally conformally flat setting). In this paper, we address a very fundamental question 
\begin{Ques}
	Let $M$ be a closed oriented spin $m$-manifold, equipped with the data  $(\ig,f)$ on $M$ (a metric and a real function), so that either $(M,\ig)$ is not conformally equivalent to $(S^m,\ig_{S^m})$ or $f:M\to\R$ is not a constant. Whether or not the set of all solutions to the spinorial Yamabe-type PDE \eqref{SY} is compact (in the $C^1$-topology, say)?
\end{Ques}
\noindent
The case of the round sphere $(S^m,\ig_{S^m})$ and $f\equiv constant\neq0$ is exceptional since \eqref{SY} is invariant under the action of the conformal group on $S^m$, which is not compact. Let us also mention here that, in the context of the spinorial Weierstra\ss\ representation, the above question may lead us to think 
\begin{Ques}\label{Q2}
	Given a connected closed oriented surface $\Sigma$ and arbitrary data $(\ig,f)$ on $\Sigma$, is it possible to characterize a non-compact family of immersions $\Pi_i:\Sigma\to\R^3$ conformally realizing $(\ig,f)$, that is
	\[
	\Pi_i^*(\ig_{\R^3})\in[\ig] \quad \text{and} \quad H_{\Pi_i}=f, \quad \text{for all } i=1,2,\dots
	\]
	where $H_{\Pi}$ stands for the mean curvature of an immersion $\Pi$?
\end{Ques}
\begin{Rem}
Usually, a generic immersion is uniquely determined up to a rigid motion by its first fundamental form and its mean curvature function, but there are some exceptions, for instance most constant mean curvature immersions. A classical result by Bonnet states that if there exists a diffeomorphism $\Psi:\Sigma_1\to\Sigma_2$ between two closed immersed surfaces $\Sigma_1$, $\Sigma_2$ of genus zero in $\R^3$ such that $\Psi$ preserves both the metric and the mean curvature function of the surfaces, then $\Sigma_1$ and $\Sigma_2$ are congruent in $\R^3$ (i.e., they differ by a rigid motion). Note that in Question \ref{Q2} we are not assuming that the immersed surfaces are isometric, which is a critical hypothesis of Bonnet's result.
\end{Rem}

Generally speaking, the answer to Question \ref{Q2} is definitely no. On the one hand, for any immersion $\Pi$ of a compact surface in $\R^3$, one must have $H_\Pi>0$ somewhere. This means that the function $f$ cannot be arbitrarily taken. And in fact there are further obstructions, at least on the sphere $\Sigma=S^2$. Indeed, the mean curvature $H_\Pi$ of a conformal immersion $\Pi: S^2\to\R^3$ must satisfy
\[
\int_{S^2}V(H_\Pi)d\vol_{\Pi^*(\ig_{\R^3})}=0
\] 
for any conformal vector field $V$ on $S^2$, see \cite{AHA}. In particular, if $x_3:S^2\to\R$ stands for the third component of the standard inclusion of $S^2$ in $\R^3$, then for any $\vr\neq0$ the function $f(x)=1+\vr x_3$ cannot be realized as the mean curvature of a conformal immersion $S^2\to\R^3$. On the other hand, it is well-known that the round sphere is the only possible shape of an immersed closed CMC surface in $\R^3$ having genus $0$ (see Hopf \cite{Hopf}). Therefore, the questions which concern us here are only interesting in the case where the solution set of Eq. \eqref{SY} is non-empty and having rich characterization to reflect the geometric interpretations.

Let us point out that a similar non-compactness question has been raised for the classical Yamabe problem, which is well-known as the Compactness Conjecture, see \cite{Schoen1991}. And such conjecture has been verified up to dimension 24 and disproved for dimensions $m\geq25$, see \cite{Brendle, BrMa, KMS}. So far, to the best of our knowledge, there is no result characterizing the compactness or non-compactness of the solution set for Eq. \eqref{SY}. One of the reasons is that, since the validity of the strict inequality in \eqref{BHL-inequ} is still open, the solvability of Eq. \eqref{SY} is far from complete. Moreover, it is also not even clear if the positive mass theorem (or its variants) can be employed to the study of Eq. \eqref{SY} as the Schoen-Yau positive energy theorem does for the classical Yamabe problem. 
    
In this paper, we intend to construct specific geometric data on a Riemannian spin manifold such that the set of solutions to Eq. \eqref{SY} fails to be compact. 
To be more precise, we will focus on the case $M=S^m$ and attack the problem from two perspectives. For starters, let us take $f\equiv 1$ in \eqref{SY} and consider the effects of the background metric $\ig$. In this case, we are facing with the equation
    \begin{\equ}\label{SY-Sm-0}
    	D_\ig\psi=|\psi|_\ig^{\frac2{m-1}}\psi \quad \text{on } (S^m,\ig)
    \end{\equ}
where $\ig$ is not conformally related to the round metric. It is of particular interest since the integral $\int_{S^m}|\psi|_\ig^{\frac{2m}{m-1}}d\vol_\ig$ of a solution gives an upper bound of the B\"ar-Hijazi-Lott invariant. Hence, it would be interesting if one can derive a conformal spectral estimate for the Dirac operator $D_\ig$ so that \eqref{BHL-inequ} is a strict inequality. Another perspective is to fix  $\ig=\ig_{S^m}$ (that is the canonical round metric) and to consider the problem with a non-constant function $f:S^m\to\R$, so that the effects of the external potential function can be detected. This leads us to consider the equation
\begin{\equ}\label{SY-H-m-sphere}
	D_{\ig_{S^m}}\psi=f(x)|\psi|_{\ig_{S^m}}^{\frac2{m-1}}\psi \quad \text{on } (S^m,\ig_{S^m})
\end{\equ}
where $f\not\equiv constant$. In this setting, when $m=2$, it is of geometric interest to show the existence of a non-compact collection of immersed spheres in $\R^3$ with a prescribed mean curvature function $f$.

Our first main result reads as
    
    \begin{Thm}\label{main thm0}
    For $k\geq1$ and $m\geq 4k+2$. There exists a Riemannian metric $\ig$ on $S^m$ of class $C^k$ and a sequence of spinors $\{\psi_i\}_{i=1}^\infty\subset C^1(S^m,\mbs(S^m))$ with the following properties:
    \begin{itemize}
    	\item[$(1)$] $\ig$ is not locally conformally flat;
    	
    	\item[$(2)$] $\psi_i$ is a nontrivial solution of the equation \eqref{SY-Sm-0}	for all $i\in\N$;
    	
    	\item[$(3)$] $\displaystyle \int_{S^m}|\psi_i|_\ig^{\frac{2m}{m-1}}d\vol_{\ig}<\big(\frac m2\big)^m\om_m$ for all $i\in\N$, and 
    	\[
    	\lim_{i\to\infty}\int_{S^m}|\psi_i|_\ig^{\frac{2m}{m-1}}d\vol_{\ig}=\big(\frac m2\big)^m\om_m;
    	\]
    	
    	\item[$(4)$] $\sup_{S^m}|\psi_i|_\ig\to+\infty$ as $i\to\infty$.
    \end{itemize}
Moreover, the strict inequality in \eqref{BHL-inequ} holds true for $\ig$, i.e.
\[
\lm_{min}^+(S^m,\ig,\sa)<\frac m2 \om_m^{\frac1m}.
\]
   \end{Thm}

\begin{Rem}
\begin{itemize}
	\item[(1)] Let us point out here that, following from our construction, the metric $\ig$ in Theorem \ref{main thm0} cannot be smooth. But the above result indicates that higher the dimension is better the regularity of the metric $\ig$ will be. 
	
	\item[(2)] Theorem \ref{main thm0} provides an example of non-smooth metric $\ig$ on a spin manifold such that the strict inequality in \eqref{BHL-inequ} holds. This is the first result of this kind in the study of B\"ar-Hijazi-Lott invariant.
	
	\item[(3)] Comparing Theorem \ref{main thm0} with the results of the Yamabe problem, it would be natural to expect a compactness result for Eq. \eqref{SY-Sm-0} in low dimensions. And it is also interesting to see if the non-compactness results hold for some $C^\infty$ smooth background metric. 
\end{itemize}
\end{Rem}

\medskip

Our next result is concerned with the external potential function $f$.

\begin{Thm}\label{main thm1}
	For every $m\geq2$, there exists a non-constant function $f\in C^2(S^m)$, $f>0$, and a sequence of spinors $\{\psi_i\}_{i=1}^\infty\subset C^1(S^m,\mbs(S^m))$ with the following properties:
	\begin{itemize}
		\item[$(1)$] $\psi_i$ is a nontrivial solution of the equation \eqref{SY-H-m-sphere} for all $i\in\N$;
		
		\item[$(2)$] $|\psi_i|_{\ig_{S^m}}>0$ on $S^m$ and there holds
		\[
		\lim_{i\to\infty}\int_{S^m}f(x)|\psi_i|_{\ig_{S^m}}^{\frac{2m}{m-1}}d\vol_{\ig_{S^m}}=\Big(\frac m2\Big)^m\om_m
		\]
		and
		\[
		\lim_{i\to\infty}\int_{S^m}f(x)^2|\psi_i|_{\ig_{S^m}}^{\frac{2m}{m-1}}d\vol_{\ig_{S^m}}=\Big(\frac m2\Big)^m\om_m;
		\]
		
		\item[$(3)$] $\sup_{S^m}|\psi_i|_{\ig_{S^m}}\to+\infty$ as $i\to\infty$.
	\end{itemize}
\end{Thm}

\begin{Rem}
\begin{itemize}
	\item[(1)] Theorem \ref{main thm1} (2) has its own geometric meaning. In fact, in dimension $m=2$, we can introduce a conformal metric $\ig_i=|\psi_i|_{\ig_{S^2}}^4\ig_{S^2}$ on $S^2$, for each $i$. Then, due to the conformal covariance of the Dirac operator (cf. \cite{Friedrich00, Ginoux}), we see that there is a spinor field $\va_i$ on $(S^2,\ig_i)$ such that
	\[
	D_{\ig_i}\va_i=f(x)\va_i \quad \text{and} \quad |\va_i|_{\ig_i}\equiv1.
	\]
	Hence, by the spinorial Weierstra\ss\ representation, there is an isometric immersion $\Pi_i:(S^2,\ig_i)\to(\R^3,\ig_{\R^3})$ with mean curvature $H_{\Pi_i}=f$. Furthermore, since the pull-back of the Euclidean volume form under this immersion is $\Pi_i^*(d\vol_{\ig_{\R^3}})=|\psi_i|_{\ig_{S^2}}^4d\vol_{\ig_{\R^2}}$, the associated Willmore energy $W(\Pi_i)$ for this immersion satisfies
	\[
	W(\Pi_i)=\int_{S^2}f(x)^2|\psi_i|_{\ig_{S^2}}^4d\vol_{\ig_{\R^2}}<8\pi
	\]
	for all $i$ (large enough). Due to Li-Yau's inequality \cite[Theorem 6]{LY82}, the immersion $\Pi_i$ covers points in $\R^3$ at most once. Hence $\Pi_i$ is actually an embedding.
	
	\item[(2)] Theorem \ref{main thm1} provides a positive answer to Question \ref{Q2}. 
Indeed, let us consider the family of immersions 
	\[
	\ci=\big\{\Pi:S^2\to\R^3:\, \Pi \text{ conformally realizes } (\ig_{S^2},f)\big\}
	\]
	and discuss its compactness (say, whether the images of $S^2$ via elements of $\ci$ form a compact collection of surfaces in $\R^3$), we find that $\{\Pi_i\}\subset \ci$ and $W(\Pi_i)\to4\pi$ as $i\to\infty$. Notice that an immersion $\Pi:\Sigma\to\R^3$ of a Riemann surface $\Sigma$ satisfies $W(\Pi)=4\pi$ if and only if $\Pi(\Sigma)$ is the round sphere. Hence, we see that $\ci$ cannot be compact since the limit of $\Pi_i(S^2)$ (even if it exists) will not realize the non-constant function $f$ as the mean curvature.
\end{itemize}
\end{Rem}

Now as an immediate consequence of the above remark, we have 
\begin{Cor}
There exists a non-constant function $f\in C^2(S^2)$ such that the family 
\[
\ce=\big\{ \Pi:S^2\to\R^3 \text{ is an embedding}:\, \Pi \text{ conformally realizes } (\ig_{S^2},f) \big\}
\]
fails to be compact in the sense that $\{\Pi(S^2):\, \Pi\in\ce\}$ is not compact in $\R^3$.
\end{Cor}

Let us sketch the main steps involved in the proofs of the Theorems \ref{main thm0} and \ref{main thm1}. In Section \ref{sec Preliminaries}, after introducing some basic concepts and notations from the spin geometry, we will reformulate our problems and work on $\R^m$ instead of $S^m$ via stereographic projection. Our goal is to construct solutions to the spinorial Yamabe-type PDEs \eqref{SY-Sm-0} and \eqref{SY-H-m-sphere} on $(\R^m,\tilde\ig)$ and $(\R^m,\ig_{\R^m})$ respectively, where either $\tilde\ig=\ig_{\R^m}+\vr\tilde\ih$ is a perturbation of the Euclidean metric or $f(x)=1+\vr \tilde H(x)$ is a perturbation from constant. In Section \ref{sec:abstract results}, we set up a perturbative variational framework so that we can reduce the our problems to a kind of finite dimensional bifurcation problem. This idea has been employed for the study of classical Yamabe problem (see, e.g., \cite{AM, AM99, BM, CY91}). Here, unlike the scalar cases, the finite dimensional problem associated to the spinorial Yamabe-type PDEs is degenerate, that is, any critical point of the main term of the reduced functional is not isolated, and the collection of these critical points appear as critical manifolds of positive dimension. Thus, it is not clear whether critical points of the reduced
functional create true solutions of the original problems. For this reason, the abstract framework in \cite{AM, BM} can not be
implemented in a straightforward manner, and somehow a delicate handling is required. In the subsequent sections, i.e., Sections \ref{sec:blow-up} and \ref{sec:blow-up2}, we check the hypothesis of the abstract framework in the two cases of our main problems mentioned above, and complete the proofs of the main results. The Appendix contains some technical computations.

	\section{Preliminaries}\label{sec Preliminaries}
	
	\subsection{Projecting the problems to $\R^m$}\label{subsec:conformal}
	
	Let us first consider Eq. \eqref{SY-Sm-0} and rewrite it in a more precise manner as
	\begin{\equ}\label{SY-m-sphere}
		D_{\ig_\vr}\psi=|\psi|_{\ig_\vr}^{\frac2{m-1}}\psi \quad \text{on } S^m,
	\end{\equ}
	in which the metric $\ig_\vr$ is a perturbation from the canonical one on $S^m$. Using the stereographic projection $\pi_P:S^m\setminus\{P\}\to\R^m$  (for an arbitrarily fixed $P\in S^m$), we obtain the following one-to-one correspondence between $\ig_\vr$ on $S^m\setminus\{P\}$ and a metric $\tilde\ig_\vr$ on $\R^m$:
	\begin{\equ}\label{the metric g}
		\tilde\ig_\vr=\mu^{-2}\cdot(\pi_P^{-1})^*\ig_\vr, \quad \mu(x)=\frac2{1+|x|^2}, \ x\in\R^m.
	\end{\equ}
	Clearly, if $\tilde\ig_\vr=\ig_{\R^m}$ is the canonical Euclidean metric, then the metric $\ig_\vr$ on $S^m\setminus\{P\}$ can be extended globally to the standard round metric. In what follows, we assume that $\tilde\ig_\vr$ takes the form $\tilde\ig_\vr=\ig_{\R^m}+\eps \tilde\ih$ where $\tilde\ih$ is a smooth symmetric bilinear form on $\R^m$. In particular, let us consider a specific situation
	\begin{\equ}\label{metric form}
		\tilde\ig_\vr(x)=\diag\big(\tilde\ig_{11}(x),\dots,\tilde\ig_{mm}(x) \big)
		\quad \text{with} \quad 
		\tilde\ig_{ii}(x)=1+\eps\tilde\ih_{ii}(x), 
	\end{\equ}
	where $\tilde\ih_{ii}:\R^m\to\R$, $i=1,\dots,m$, are smooth functions.  Let us point out here that, for a general choice of $\tilde\ih$, the pull-back metric $\ig_\vr$ on $S^m\setminus\{P\}$ may be discontinuous at the point $P$. Hence, in order to extend $\ig_\vr$ globally on $S^m$, it is natural to require the entries $\tilde\ih_{ii}$, $i=1,\dots,m$, and their derivatives behave ``nicely" at infinity. 
	
	We also mention that the Eq. \eqref{SY-m-sphere} on $S^m$ is equivalent to an equation on $\R^m$ by conformal equivalence. More precisely, the equation $D_\ig\psi=|\psi|_\ig^{\frac2{m-1}}\psi$ on a spin manifold $(M,\ig)$ is invariant under conformal changes of the metric. In fact, let $\bar\ig=e^{2u}\ig$ for some function $u$ on $M$, there is an isomorphism of vector bundles $F:\mbs(M,\ig)\to\mbs(M,\bar\ig)$
	(here $\mbs(M,\ig)$ and $\mbs(M,\bar\ig)$ are spinor bundles on $M$ with respect to the metrics $\ig$ and $\bar\ig$, respectively) which is a fiberwise isometry such that
	\[
	D_{\bar\ig}\big( F(\va)\big)= F\big( e^{-\frac{m+1}2u}D_\ig(e^{\frac{m-1}2u}\va) \big)
	\]
	for $\va\in C^\infty((M,\ig),\mbs(M,\ig))$ (for more detailed definitions and facts about Clifford algebras, spin structures on manifolds and Dirac operators, please consult \cite{Friedrich00, LM}). Thus, when $\psi$ is a solution to the equation $D_\ig\psi=|\psi|_\ig^{\frac2{m-1}}\psi$  on $(M, \ig)$, then $\va:= F(e^{-\frac{m-1}2u}\psi) $ satisfies the same equation on $(M, \bar\ig)$: $D_{\bar\ig}\va=|\va|_{\bar\ig}^{\frac2{m-1}}\va$ on $(M, \bar\ig)$.
	
	\medskip
	
	Applying the above observation to Eq. \eqref{SY-m-sphere} with $M=S^m$ and using \eqref{the metric g}-\eqref{metric form}, we find that if $\psi\in C^1(S^m,\mbs(S^m))$ is a solution then $\va=\mu^{\frac{m-1}2}F(\psi\circ\pi_{P}^{-1})$ satisfies the equation
	\begin{\equ}\label{reduced Dirac problem 1}
		D_{\tilde\ig_\vr}\va=|\va|_{\tilde\ig_\vr}^{\frac2{m-1}}\va \quad \text{on } \R^m.
	\end{\equ}
    Conversely, by the regularity theorem and the removal of singularities theorem for Dirac equations on spin manifolds (see \cite[Appendix]{Isobe11} and \cite[Theorem 5.1]{Ammann2009}), if $\va$ is a solution  to Eq. \eqref{reduced Dirac problem 1} and $\va\in L^{\frac{2m}{m-1}}(\R^m,\mbs(\R^m))$ then it corresponds to a global $C^1$-solution $\psi$ to Eq. \eqref{SY-Sm-0} on $S^m$. Therefore, the study of Eq. \eqref{SY-m-sphere} is equivalent to the study of Eq. \eqref{reduced Dirac problem 1}.

	\medskip
	
	Now, to characterize the metric $\ig_\vr$, let us set
	\begin{\equ}\label{h-expansion}
		\tilde\ih(x)=\sum_{i=1}^{+\infty}a_i \ih(x-x_i) 
	\end{\equ} 
	where $a_i\in\R$, $|x_i|\to+\infty$ as $i\to+\infty$, and $h$ is a smooth symmetric matrix function with compact support. Roughly speaking, with this choice of $\tilde\ih$ in the definition of metric $\tilde\ig_\vr$, the Dirac operator in \eqref{reduced Dirac problem 1} becomes 
	\begin{\equ}\label{Dirac-operator-expansion}
	D_{\tilde\ig_\vr}=D_{\ig_{\R^m}}+R(\vr,x,\tilde\ih,\nabla)
	\end{\equ}
	where $R(\vr,x,\tilde\ih,\nabla)$ is a suitable perturbation term. In this way, we expect that Eq. \eqref{reduced Dirac problem 1} shall be handled by means of a perturbation method in nonlinear analysis. When $\tilde\ih$ consists of only a finite number of terms, the existence problem of Eq. \eqref{reduced Dirac problem 1} has been firstly treated in \cite{Isboe-Xu21}. In particular, a very specific construction of the matrix function $\ih$ has been introduced in \cite{Isboe-Xu21} so that the effect of the perturbation term in \eqref{Dirac-operator-expansion} can be explicitly computed from a variational point of view. Here, for the sake of completeness, we present the very formulation for the matrix $\ih$ in \eqref{h-expansion} as follows.
	
	\begin{Def}\label{def k-elementary}
		\it Given a smooth $m\times m$ diagonal matrix function
		\[
		\ih(x)=\diag\big(\ih_{11}(x),\dots, \ih_{mm}(x) \big) \quad \text{for } x\in \R^m,\ m\geq2,
		\]
		and a point $\xi=(\xi_1,\dots,\xi_m)\in\R^m$. For $k\in\{1,\dots,m\}$ and $p\in[1,\infty)$, we say that $\ih$ is  $(k,p)$-elementary at $\xi$, if 	$\xi\not\in\supp \ih_{kk}$ and, for $x=(x_1,\dots,x_m)\in\R^m$ close to $\xi$ and $i\neq k$, 
		\[
		\ih_{ii}(x)=\ih_{ii}(\xi)+c_i(x_i-\xi_i)+c_k(x_k-\xi_k)+o(|x-\xi|^p)
		\]
		where $c_i\in\R$, $i=1,\dots,m$, are constants with particularly $c_k\neq0$. Moreover, if the $o(|x-\xi|^p)$ term vanishes identically in the above local expansion of $\ih_{ii}$'s, then we say $\ih$ is $(k,\infty)$-elementary at $\xi$. In this way, we call  $p\in[1,\infty)\cup\{\infty\}$ the remainder exponent of $\ih$ at $\xi$.
	\end{Def}

	\begin{Rem}
	Let us present here a simple example of $(1,p)$-elementary matrix at the origin, in dimension $3$: 
	\[
	\ih(x)=\begin{pmatrix}
		\,0 & 0 & 0 \\
		\,0 & a+c_1x_1+c_2x_2 & 0 \\
		\,0 & 0 & b + c_1x_1 +c_3 x_3 \, 
	\end{pmatrix} + o(|x|^p)
	\]
	for $|x|<r$, where $a,\,b,\, c_1,\,c_2,\,c_3\in\R$ are real constants with particularly $c_1\neq0$. This very specific definition is first introduced in \cite{Isboe-Xu21} for the study of multiple solutions for the spinorial Yamabe-type problems. More examples and a brief explanation of such $(k,p)$-elementary matrices have been given in \cite[Appendix]{Isboe-Xu21}. We mention here that the main reason we introduce those $(k,p)$-elementary matrices lies in Proposition \ref{prop hat-Phi-ih}, where we find such matrices are surprisingly compatible with the perturbed Dirac operator \eqref{Dirac-operator-expansion} and they guarantee the implementation of our abstract result in Section \ref{sec:abstract results}. 
	\end{Rem}

	It can be seen from Definition \ref{def k-elementary} that ``{\it elementary}" matrix is a local concept. In the sequel, if it is clear from the context to which dimension we refer, we will simply use the name ``elementary matrix" to designate a member $\ih$ (without specifying its tag numbers $k$ and the location point $\xi$). In order to classify the perturbation term in \eqref{h-expansion}, let us set
	\[
		\ch(p)=\left\{ 	\tilde\ih(\cdot)=\sum_{i=1}^{+\infty}a_i \ih(\cdot-x_i) \left|\,
		\aligned
		& \ih \text{ is a compactly supported elementary matrix}\\
		&\text{with remainder exponent }p, \\
		& \{a_i\}\subset\R \text{ and } \sum_{i=1}^\infty|a_i|^\tau<+\infty, \text{ for some } \tau>1, \\
		& \{x_i\}\subset\R^m \text{ and } |x_j-x_i|>4\diam(\supp\ih) \text{ for } i\neq j
		\endaligned
		\right.\right\}.
	\]
Then Theorem \ref{main thm0} is nothing but a direct consequence of the following result.

	\begin{Thm}\label{main thm}
		Let $p\in[2,\infty)\cup\{\infty\}$, $k\geq1$ and $m\geq4k+2$. There exist $\tilde\ih\in\ch(p)$ and $\vr_0>0$ such that for every $\vr\in(-\vr_0,\vr_0)\setminus\{0\}$ the metric $\ig_\vr$ in \eqref{the metric g}-\eqref{h-expansion} is of class $C^k$ on $S^m$, and the following properties hold:
		\begin{itemize}
			\item[$(1)$] $\|\ig_\vr-\ig_{S^m}\|_{C^k}\to0$ as $\vr\to0$,
			
			\item[$(2)$] Eq. \eqref{SY-m-sphere} possesses a sequence of solutions $\{\psi_\vr^{(i)}\}_{i=1}^{\infty}$ satisfying $\|\psi_\vr^{(i)}\|_{L^\infty(S^m)}\to+\infty$ as $i\to+\infty$,
			
			\item[$(3)$] there holds 
			\[
			\int_{S^m}|\psi_\vr^{(i)}|_{\ig_\vr}^{\frac{2m}{m-1}}d\vol_{\ig_\vr}=\big(\frac m2\big)^m\om_m+C_{i,m}a_i^2\vr^2+o(a_i^2\vr^2)
			\]
			where $C_{i,m}<0$ is a negative constant depending only on $\tilde\ih$, the dimension $m$ and $i\in\N$. 
		\end{itemize} 
	\end{Thm}
	
	\medskip
	
	Next, in order to study Eq. \eqref{SY-H-m-sphere}, let us focus on the case where the function $f$ takes the form $f(x)=1+\vr \tilde H(x)$, i.e.
	\begin{\equ}\label{SY-H-vr}
		D_{\ig_{S^m}}\psi=(1+\vr\tilde H(x))|\psi|_{\ig_{S^m}}^{\frac2{m-1}}\psi\quad \text{on } (S^m,\ig_{S^m})
	\end{\equ}
	with $\vr\neq0$ and
	some $\tilde H:S^m\to\R$ at least being H\"older continuous. As before, denote by $\pi_{p_0}:S^m\setminus\{p_0\}\to\R^m$ the stereographic projection from $p_0$ (this point will be fixed later according to our choice of $\tilde H$), we have $(\pi_{p_0}^{-1})^*\ig_{S^m}=\mu^2\ig_{\R^m}$. And then, via the conformal transformation, Eq. \eqref{SY-H-m-sphere} can be converted  to
	\begin{\equ}\label{reduced Dirac problem 2}
		D_{\ig_{\R^m}}\va=(1+\vr \tilde K(x))|\va|_{\ig_{\R^m}}^{\frac2{m-1}}\va \quad \text{on } \R^m
	\end{\equ}
where $\tilde K(x)=\tilde H(\pi_{p_0}^{-1}(x))$. We remark here that we shall write it simply $x$ for the argument of a function when no confusion can arise.

Similarly to the way we handle Eq. \eqref{reduced Dirac problem 1}, let us consider a situation where the function $\tilde H$ can be decomposed into a series of components such that each component generates a solution to Eq. \eqref{reduced Dirac problem 2}. In order to do so, for a continuously differentiable function $H$ on $S^m$ (which plays the role of an individual component of $\tilde H$), let us simply denote $Crit[H]$ the critical set of $H$. For later use, we assume the following two standing conditions on $H$:
\begin{itemize}
	\item[(H-1)] $H\in C^2(S^m)$ is a Morse function such that $\De_{\ig_{S^m}}H(p)\neq 0$ for $p\in Crit[H]$.
	\item[(H-2)] $H$ satisfies that
	\[
	\sum_{p\in Crit[H],\ \De_{\ig_{S^m}}H(p)<0} (-1)^{\mfm(H,p)}\neq (-1)^m
	\]
	where $\mfm(H,p)$ is the Morse index of $H$ at $p\in Crit[H]$.
\end{itemize}
Here we mention that condition $(\text{H-2})$ is the well-known index counting condition which was first introduced in the scalar curvature problem in \cite{BC91, CGY93}. 

Then we collect the following family of continuous functions on $S^m$
\[
	\msh=\left\{
	\tilde H=\sum_{i=1}^\infty a_i H\big(\pi_{p_0}^{-1}(\pi_{p_0}(\cdot)-z_i)\big)   \left| \,
	\aligned
	& H \text{ satisfies the conditions $\text{(H-1)}$ and $\text{(H-2)}$},   \\
	& p_0\in Crit[H] \text{ and } \De_{\ig_{S^m}} H(p_0)>0, \\
	& \{a_i\}\subset\R \text{ and } \sum_{i=1}^\infty|a_i|<+\infty, \\
	& \{z_i\}\subset\R^m \text{ and }
	|z_i-z_j|>1 \text{ for } i\neq j
	\endaligned \right.
	\right\}.
\]
It is clear that, when $H\in C^2(S^m)$, $H\circ\pi_{p_0}^{-1}$ defines a $C^2$-function on $\R^m$ and $\lim_{|y|\to\infty}H\circ\pi_{p_0}(y)=H(p_0)$. The function $H\big(\pi_{p_0}^{-1}(\pi_{p_0}(\cdot)-z_i)\big):S^m\setminus\{p_0\}\to\R$ can be viewed as a translation of $H$ on $S^m$ with $H(p_0)$ being fixed. Hence the above family describes a function that is (approximately) concentrated on the points $\pi_{p_0}^{-1}(z_i)\in S^m\setminus\{p_0\}$, $i=1,2,\dots$, and is well-defined on $S^m$. We remark that the elements in $\msh$ is not necessarily differentiable at $p_0$ and, as was indicated in its geometric background, the function $1+\vr \tilde H$ plays a role of mean curvature. Hence, one may expect $\tilde H$ to have certain regularity at the point $p_0$. With all these in mind, let us present the following result that explains Theorem \ref{main thm1}.

\begin{Thm}\label{main thm-H}
	For every $m\geq2$, there exist $\tilde H\in\msh\cap C^2(S^m)$ and $\vr_0>0$ such that for $\vr\in(-\vr_0,\vr_0)\setminus\{0\}$ the following properties hold:
	\begin{itemize}
		\item[$(1)$] Eq. \eqref{SY-H-vr} possesses a sequence of solutions $\{\psi_\vr^{(i)}\}_{i=1}^{\infty}$ satisfy $\|\psi_\vr^{(i)}\|_{L^\infty(S^m)}\to+\infty$ as $i\to+\infty$,
		
		\item[$(2)$] $|\psi_\vr^{(i)}|_{\ig_{S^m}}>0$ on $S^m$ provided that $|\vr|$ is small, moreover,
		\[
		\lim_{i\to\infty}\int_{S^m}(1+\vr\tilde H(x))|\psi_\vr^{(i)}|_{\ig_{S^m}}^{\frac{2m}{m-1}}d\vol_{\ig_{S^m}}=\Big(\frac m2\Big)^m\om_m
		\] 
		and
		\[
		\lim_{i\to\infty}\int_{S^m}(1+\vr\tilde H(x))^2|\psi_\vr^{(i)}|_{\ig_{S^m}}^{\frac{2m}{m-1}}d\vol_{\ig_{S^m}}=\Big(\frac m2\Big)^m\om_m.
		\] 
	\end{itemize}
\end{Thm}


\medskip

We end this subsection by pointing out that the following equation
\begin{\equ}\label{unperturbed equ}
	D_{\ig_{\R^m}}\psi=|\psi|_{\ig_{\R^m}}^{\frac2{m-1}}\psi
	\quad \text{on }\R^m.
\end{\equ}
can be viewed as the unperturbed equation of both  Eq. \eqref{reduced Dirac problem 1} and Eq. \eqref{reduced Dirac problem 2}. Hence, in the sequel, our framework will be build upon the study of Eq. \eqref{unperturbed equ} and its Euler-Lagrange functional
\begin{\equ}\label{unperturbed functional}
	\cj_0(\psi)=\frac12\int_{\R^m}(\psi, D_{\ig_{\R^m}}\psi)_{\ig_{\R^m}} d\vol_{\ig_{\R^m}}-\frac{m-1}{2m}\int_{\R^m}|\psi|_{\ig_{\R^m}}^{\frac{2m}{m-1}}d\vol_{\ig_{\R^m}}
\end{\equ}
where $(\cdot,\cdot)_{\ig_{\R^m}}$ and $|\cdot|_{\ig_{\R^m}}$ are the canonical hermitian product and its induced metric on the spinor bundle $\mbs(\R^m)$.
	
\subsection{Configuration spaces}
	
	To treat  Eq. \eqref{reduced Dirac problem 1} and \eqref{reduced Dirac problem 2} from a variational point of view, it is necessary to set up a functional framework. Suitable function spaces are $H^{\frac12}(M,\mbs(M))$ and $\msd^{\frac12}(\R^m,\mbs(\R^m))$ of spinor fields which are introduced in \cite{Isobe11, Isobe13}. For completeness, we give the definitions as follows. 
	
	Recall that the Dirac operator $D_\ig$ on a compact spin manifold $(M,\ig)$ is self-adjoint on $L^2(M, \mbs(M))$ and has compact resolvents (see \cite{Friedrich00, LM}). Particularly, there exists a complete orthonormal basis $\psi_1,\psi_2, . . .$ of the Hilbert space $L^2(M, \mbs(M))$ consisting of the eigenspinors of $D_\ig$: $D_\ig\psi_k = \lm_k\psi_k$. Moreover, $|\lm_k| \to\infty$as $k \to\infty$.
	
	For $s\geq0$, we define the operator $|D_\ig|^s:L^2(M, \mbs(M))\to L^2(M, \mbs(M))$ by
	\[
	|D_\ig|^s\psi=\sum_{k=1}^\infty |\lm_k|^s\al_k\psi_k,
	\]
	for $\psi=\sum_{k=1}^\infty\al_k\psi_k\in L^2(M, \mbs(M))$. Denoted by $H^s(M,\mbs(M))$ the domain of $|D_\ig|^s$, then it is clear that $\psi=\sum_{k=1}^\infty\al_k\psi_k\in H^s(M,\mbs(M))$ if and only if $\sum_{k=1}^\infty|\lm_k|^{2s}|\al_k|^2<\infty$. And hence $H^s(M,\mbs(M))$ coincides with the usual Sobolev space of order $s$, that is $W^{s,2}(M,\mbs(M))$ (cf. \cite{Adams, Ammann}). Furthermore, we can equip $H^s(M,\mbs(M))$ with the inner product
	\[
	\inp{\psi}{\va}_{s,2}:=\real(|D_\ig|^s\psi,|D_\ig|^s\va)_2+\real(\psi,\va)_2
	\]
	where $(\cdot,\cdot)_2$ is the $L^2$-inner product on spinors and the induced norm $\|\cdot\|_{s,2}$. In this paper, we are mainly concerned with the space $H^\frac12(M,\mbs(M))$ for $M=S^m$. Notice that the spectrum of $D_{\ig_{S^m}}$ on $S^m$ is bounded away from $0$ and one checks easily that $\|\psi\|_{1/2}=\big||D_{\ig_{S^m}}|^\frac12\psi\big|_2$
	defines an equivalent norm on $H^\frac12(S^m,\mbs(S^m))$.

	To simplify notations, throughout this paper, we shall denote $L^q:=L^q(M,\mbs(M))$ with the norm $|\psi|_q^q= \int_M |\psi|^qd\vol_{\ig} $ for $q\geq1$ and denote $2^*=\frac{2m}{m-1}$  the critical Sobolev exponent of the embedding $H^{1/2}(M,\mbs(M))\hookrightarrow L^q(M,\mbs(M))$ for $1\leq q\leq 2^*$.
	
	On $\R^m$, a similar function space will also be useful in our argument. Let us denote by $\msd^{\frac12}(\R^m,\mbs(\R^m))$ the set of spinor fields $\psi$ on $\R^m$ such that $\big||D_{\ig_{\R^m}}|^{1/2}\psi\big|_2^2
	<\infty$ and $|\psi|_{2^*}<\infty$ with norm $\|\psi\|:=\big||D_{\ig_{\R^m}}|^{1/2}\psi\big|_2+|\psi|_{2^*}$. Here, $|D_{\ig_{\R^m}}|^{1/2}$ is defined via the Fourier transformation: $\msf(|D_{\ig_{\R^m}}|^{1/2}\psi)(\xi)=|\xi|^{1/2}\msf(\psi)(\xi)$
	and $\big||D_{\ig_{\R^m}}|^{1/2}\psi\big|_2:=\big| |\cdot|^{1/2}\msf(\psi) \big|_2$. Notice that $\msd^{\frac12}(\R^m,\mbs(\R^m))$ is isomorphic to $H^\frac12(S^m,\mbs(S^m))$ via the stereographic projection. The dual space of $\msd^{\frac12}(\R^m,\mbs(\R^m))$ will be denoted by $\msd^{-\frac12}(\R^m,\mbs(\R^m))$.
	
\subsection{Geometric preliminaries and expansion of the perturbed functional}

In this part, we shall collect some basic results that will enable us to expand the energy functional $\cj_\vr$ for \eqref{reduced Dirac problem 1} with respect to the small parameter $\vr$ in \eqref{the metric g}-\eqref{metric form}. In order to carry this out, we need to identify Dirac operators and spinor fields on $\R^m$ with respect to the canonical Euclidean metric $\ig_{\R^m}$ and the new metric $\tilde\ig$. The construction by Bourguignon and Gauduchon \cite{BG} provides us such a necessary identification.

To begin with, let us denote $\odot_+^2(\R^m)$ the space of  positive definite symmetric bilinear forms, i.e. metrics on $\R^m$. Then, for small values of $\vr>0$, we consider the identity map
\[
\aligned
\id: (\R^m,\ig_{\R^m})&\to (\R^m,\tilde\ig)  \\
x\ &\mapsto\  x
\endaligned
\]
and the map
\[
\aligned
\tilde G :\R^m \ &\to \ \odot_+^2(\R^m) \\
x\ &\mapsto \ \tilde G_x:=(\tilde{\ig}_{ij}(x))_{ij}
\endaligned
\]
which associates to a point $x\in\R^m$ the matrix of the coefficients of the metric $\tilde{\ig}$ at this point, expressed in the basis $\pa_i=\frac{\pa}{\pa x^i}$, $i=1,\dots,m$. Notice that $\tilde G_x\in\odot_+^2(\R^m)$, hence $G_x$ is invertible and there is a unique matrix $B_x\in\odot_+^2(\R^m)$ such that $B_x^2=\tilde G_x^{-1}$. Let $b_{ij}(x)$, $i,j=1,\dots,m$, be the entries of $B_x$, we have
\[
\aligned
B_x: (T_x\R^m\cong\R^m,\ig_{\R^m}) &\to (T_x\R^m,\tilde\ig_x) \\
v=\sum_k v_k\pa_k\ &\mapsto \ B_x(v):=\sum_j\big( \sum_k b_{jk}(x)v_k\big)\pa_j
\endaligned
\]
defines an isometry for each $x\in\R^m$. As the matrix $B_x$ depends smoothly on $x$, we obtain an isomorphism of $SO(m)$-principal bundles:
\begin{displaymath}
	\xymatrix{
		P_{SO}(\R^m,\ig_{\R^m}) \ar[r]^{ \ \eta}  \ar[d] & P_{SO}(\R^m,\tilde\ig) \ar[d] \\
		\R^m\ar[r]^{\ \ \id} & \R^m
	}
\end{displaymath}
where $\eta\{v_1,\dots,v_m\}=\{B(v_1),\dots,B(v_m)\}$ for an oriented frame $\{v_1,\dots,v_m\}$ on $(\R^m,\ig_{\R^m})$. Note that since the map $\eta$ commutes with the right action of $SO(m)$, it can be lifted to spin structures:
\begin{displaymath}
	\xymatrix{
		P_{Spin}(\R^m,\ig_{\R^m}) \ar[r]^{ \tilde\eta} \ar[d] & P_{Spin}(\R^m,\tilde\ig)  \ar[d]\\
		\R^m \ar[r]^{\id}  &  \R^m
	}
\end{displaymath}
And hence we obtain an isomorphism between the spinor bundles $\mbs(\R^m,\ig_{\R^m})$ and $\mbs(\R^m,\tilde{\ig})$:
\begin{\equ}\label{spinor identify}
	\aligned
	\mbs(\R^m,\ig_{\R^m}) := P_{Spin}(\R^m,\ig_{\R^m})\times_\rho \mbs_m &\longrightarrow \mbs(\R^m,\tilde{\ig}) := P_{Spin}(\R^m,\ig)\times_\rho \mbs_m  \\
	\psi=[s,\va]&\longmapsto \tilde\psi=[\tilde\eta(s),\va]
	\endaligned
\end{\equ}
where $\rho$ is the complex spinor representation and $[s,\va]$ stands for the equivalence class of $(s,\va)$ under the action of $Spin(m)$. This identifies the spinor fields.

For the Dirac operators, as was shown by \cite[Proposition 3.2]{AGHM}, the identification can be expressed in the following formula
\begin{\equ}\label{Dirac identify}
	D_{\tilde\ig}\tilde\psi=\widetilde{D_{\ig_{\R^m}}\psi}+W\cdot_{\tilde\ig}\tilde\psi+X\cdot_{\tilde\ig}\tilde\psi+\sum_{i,j}(b_{ij}-\de_{ij})\tilde\pa_i\cdot_{\tilde\ig}\widetilde{\nabla_{\pa_j}\psi}
\end{\equ}
where $\cdot_{\tilde\ig}$ denotes the Clifford multiplication with respect to the metric $\tilde\ig$,
\[
W = \frac14\sum_{\substack{i,j,k \\ i\neq j\neq k\neq i}}\sum_{\al,\bt} b_{i\al}(\pa_{\al}b_{j\bt})b_{\bt k}^{-1}\,\tilde\pa_i\cdot_{\tilde\ig} \tilde\pa_j\cdot_{\tilde\ig} \tilde\pa_k,
\]
with $b_{ij}^{-1}$ being the entries of the inverse matrix of $B$, $\tilde\pa_i=B(\pa_i)$ and
\[
X = \frac12\sum_{i,k} \tilde\Ga_{ik}^i \tilde\pa_k,
\]
with $\tilde\Ga_{ij}^k=\tilde\ig(\tilde\nabla_{\tilde\pa_i}\tilde\pa_j,\tilde\pa_k) $ being the Christoffel symbols of the second kind. 

\begin{Rem}
	On spin manifolds, since the tangent bundle is embedded in the bundle of Clifford algebra, vector fields have two different actions on spinors, i.e. the Clifford multiplications and the covariant derivatives. Here, to distinguish the two actions on a spinor $\psi$, we denote $\pa_i\cdot_{\ig_{\R^m}}\psi$ the Clifford multiplication of $\pa_i$  and $\nabla_{\pa_i}\psi$ the covariant derivative with respect to the metric $\ig_{\R^m}$ (respectively, $\tilde\pa_i\cdot_{\tilde\ig}\tilde\psi$ the Clifford multiplication of $\tilde\pa_i$ and $\tilde\nabla_{\tilde\pa_i}\tilde\psi$ the covariant derivative with respect to the metric $\tilde\ig$). For functions, we shall simply denote $\pa_i u$ for its partial derivative.
	
\end{Rem}

Now we recall some formulas given in \cite{Isboe-Xu21} which are direct consequences of some elementary computations and will be useful in our framework.

\begin{Lem}\label{basic expansions}
	Let $\tilde\ig_\vr$ be given by \eqref{the metric g}-\eqref{h-expansion}, we have
	\begin{\equ}\label{expansion det G}
		\sqrt{\det\tilde G_\vr}=1+\frac\eps2\tr\tilde\ih+\eps^2\Big(
		\frac18(\tr\tilde\ih)^2-\frac14\tr(\tilde\ih^2)\Big)+o(\eps^2),
	\end{\equ}
	\begin{\equ}\label{expansion B}
		B_\vr=I-\frac\eps2\tilde\ih+\frac{3\,\eps^2}8\tilde\ih^2+o(\eps^2)
	\end{\equ}
	and
	\begin{\equ}\label{expansion B inverse}
		B^{-1}_\vr=I+\frac\eps2\tilde\ih-\frac{\eps^2}8\tilde\ih^2+o(\eps^2).
	\end{\equ}
\end{Lem}

With the notation of Bourguignon-Gauduchon identification, for a spinor $\tilde\psi$ in $\mbs(\R^m,\tilde\ig_\vr)$, the energy functional associated to \eqref{reduced Dirac problem 1} is defined as
\begin{\equ}\label{functional dimension m}
	\cj_\vr(\tilde\psi)=\frac12\int_{\R^m}(\tilde\psi,D_{\tilde\ig_\vr}\tilde\psi)_{\tilde\ig_\vr}\,d\vol_{\tilde\ig_\vr}-\frac1{2^*}\int_{\R^m}|\tilde\psi|_{\tilde\ig_\vr}^{2^*}\,d\vol_{\tilde\ig_\vr}.
\end{\equ}
It turns out that
$\cj_\vr$ is an $\eps$-involved functional because $\tilde\ig_\vr$ is given by \eqref{the metric g}-\eqref{metric form}. Via the Bourguignon-Gauduchon identification, we have the following expansion:
\begin{Lem}\label{expansion J functional}
	Let $\tilde\ig$ be given by \eqref{the metric g}-\eqref{h-expansion}, then
	\begin{\equ}\label{expansion of cj-vr}
	\cj_\vr(\tilde\psi)=\cj_0(\psi)+\eps \Ga(\psi)+\eps^2\Phi(\psi)+o(\eps^2),
	\end{\equ}
	where 
	\[
	\cj_0(\psi)=\frac12\int_{\R^m}(\psi,D_{\ig_{\R^m}}\psi)_{\ig_{\R^m}}\,d\vol_{\ig_{\R^m}}-\frac1{2^*}\int_{\R^m}|\psi|_{\ig_{\R^m}}^{2^*}\,d\vol_{\ig_{\R^m}},
	\]
	\[
	\Ga(\psi)=\int_{\R^m}\frac{\tr\tilde\ih}2\Big[ \frac12\big(\psi,D_{\ig_{\R^m}}\psi \big)_{\ig_{\R^m}}-\frac1{2^*}|\psi|_{\ig_{\R^m}}^{2^*}\Big]
	-\frac14\sum_i\tilde\ih_{ii}\real\big(\pa_i\cdot_{\ig_{\R^m}}\nabla_{\pa_i}\psi,\psi \big)_{\ig_{\R^m}} d\vol_{\ig_{\R^m}}
	\]
	and
	\[
	\aligned
	\Phi(\psi)&=\int_{\R^m}\Big(
	\frac18(\tr\tilde\ih)^2-\frac14\tr(\tilde\ih^2)\Big)\Big[ \frac12\big(\psi,D_{\ig_{\R^m}}\psi \big)_{\ig_{\R^m}}-\frac1{2^*}|\psi|_{\ig_{\R^m}}^{2^*}\Big] \\[0.5em]
	&\qquad
	+\frac1{16}\sum_i\Big( 3\tilde\ih_{ii}^2-2(\tr\tilde\ih)\tilde\ih_{ii} \Big)\real(\pa_i\cdot_{\ig_{\R^m}}\nabla_{\pa_i}\psi,\psi)_{\ig_{\R^m}} d\vol_{\ig_{\R^m}}
	\endaligned
	\]
	for $\psi\in\msd^{\frac12}(\R^m,\mbs(\R^m))$. 
\end{Lem}

\begin{Rem}
We omit the detailed proof of Lemma \ref{expansion J functional} because the main computations are already contained in the proof of Lemma \ref{lem:inequalities} later, see Appendix \ref{proof of lemma 4.5} or \cite[Lemma 4.3]{Isboe-Xu21}.
\end{Rem}
	
	\section{Abstract settings}\label{sec:abstract results}
	
	\subsection{Lyapunov–Schmidt reduction of the functional}
	
	It is clear that, on $\msd^{\frac12}(\R^m,\mbs(\R^m))$, the functional $\cj_0$ in \eqref{unperturbed functional} is well-defined and is of $C^2$ class. We call $\cj_0$ the unperturbed functional. Its critical points are solutions to the unperturbed spinorial Yamabe equation \eqref{unperturbed equ}. Recall that Eq. \eqref{reduced Dirac problem 1} and \eqref{reduced Dirac problem 2} are perturbations from Eq.  \eqref{unperturbed equ} and it is very natural to expect that solutions of  Eq. \eqref{reduced Dirac problem 1} and \eqref{reduced Dirac problem 2} can be obtained as perturbations of critical points of $\cj_0$. In general, a well adopted Lyapunov-Schmidt reduction technique provides a powerful tool, see for instance \cite[Chapter 10]{MW} and \cite[II, 6]{Chang} where the reduced problem is compact and \cite{AB, AB98, AM} for the case that the reduced problem is non-compact.
	For completeness, following \cite{AM}, we sketch the idea as follows.
	
	Let $(\ch,\inp{\cdot}{\cdot})$ be a Hilbert space with the associated norm $\|\cdot\|:=\inp{\cdot}{\cdot}^{1/2}$. Suppose that $L_0\in C^2(\ch,\R)$ and $\Ga\in C^2(\ch,\R)$ are given. For $\eps>0$ small, we consider the perturbed functional
	\begin{\equ}\label{model problem}
		L_\eps(z)=L_0(z)+\eps \Ga(z)+o(\eps).
	\end{\equ}
	Assume that $L_0$ has a non-degenerate critical manifold $\cm\subset\ch$, that is,
	\begin{itemize}
		\item[$(A1)$] $\cm$ is a $d$-dimensional $C^2$-submanifold of $\ch$ such that $\nabla L_0(z)=0$ for all $z \in \cm$,
		
		\item[$(A2)$] $\cm$ is non-degenerate in the sense that for all $z \in\cm$, we have $T_z\cm = \ker \nabla^2 L_0 (z)$,
		
		\item[$(A3)$] $\nabla^2 L_0 (z) : \ch \to\ch$ is a Fredholm operator with index zero for all $z \in\cm$.
	\end{itemize}
	
	Once the conditions $(A1)$-$(A3)$ are satisfied, the problem of finding critical points of $L_\eps$ on $\ch$ can be reduced to the same problem on $\cm$ for a reduced functional $L_\eps^{red}$ (defined below). And critical points of $L_\eps$ will be obtained as small perturbations of critical points of $L_0$.
	
	Set $\cw_z:=T_z\cm^\bot$, where the orthogonal complement is taking with respect to $\inp{\cdot}{\cdot}$ in $\ch$. We look for critical points of $L_\eps$ in the form $u=z+w$ where $z\in\cm$ and $w\in \cw_z$. Let $P_z:\ch\to\cw_z$ be the orthogonal projection onto $\cw_z$, the Euler-Lagrange equation $\nabla L_\eps(z+w)=0$ is equivalent to
	\begin{\equ}\label{L-S reduction}
		\left\{\aligned
		&P_z\nabla L_\eps(z+w)=0  & & (\textit{auxiliary equation})\\
		&(I-P_z)\nabla L_\eps(z+w)=0 & &  (\textit{bifurcation equation}).
		\endaligned \right.
	\end{\equ}
	Then, under the conditions $(A2)$ and $(A3)$, the auxiliary equation in \eqref{L-S reduction} can be solved firstly for $w$ by applying the implicit function theorem: for arbitrary $z\in\cm$ there is a unique small solution $w=w_\vr(z)\in\cw_z$ for small values of $\eps$. Furthermore, on any compact subset $\cm_c\subset\cm$, one can have the uniform estimate:
	\begin{\equ}\label{AM esti}
		\cm_c\ni z\mapsto w_\eps(z)\in \cw_z \text{ is } C^1 \text{ and }
		\|w_\eps(z)\|,\ \|w_\eps'(z)\|= O(\eps) \text{ uniformly for }
		z\in\cm_c.
	\end{\equ}
	For the proofs of \eqref{AM esti} and more detailed explanation about this construction, see \cite[Chapter 2]{AM}.
	
	The next step is to consider the bifurcation equation in \eqref{L-S reduction}. To this end, we introduce the reduced functional $L_\eps^{red}:\cm\to\R$ by
	\[
	L_\eps^{red}(z)=L_\eps(z+w_\eps(z))
	\]
	Then we have the following theorem
	\begin{Thm}[Theorem 2.12 in \cite{AM}]\label{AM Theorem}
		Suppose $(A1)$-$(A3)$ are satisfied. Assume that for a compact subset $\cm_c\subset\cm$ and $\eps>0$ small, $L_\eps^{red}$ has a critical point $z_\eps\in\cm_c$. Then $u_\eps=z_\eps+w_\eps(z_\eps)$ is a critical point of $L_\eps$ on $\ch$.
	\end{Thm}
	
	Thanks to the uniform estimate \eqref{AM esti}, the reduced functional $L_\eps^{red}$ is well approximated in the sense that
	\begin{\equ}\label{L approximate}
		L_\eps^{red}(z)=L_0(z)+\eps\Ga(z)+o(\eps), \quad
		\nabla L_\eps^{red}(z)=\eps\nabla \Ga(z)+o(\eps).
	\end{\equ}
	and $L_0(z)$ is constant on any connected component of $\cm$. Thus if $z\in\cm$ is a non-degenerate critical point of $\Ga$ in some certain sense (for example, the local degree of $\nabla\Ga$ at $z$ is non-zero), then $z$ generates a critical point of
	$L_\eps$ on $\ch$ (see \cite{AM, Chang, MW} for details).
	
	\begin{Rem}
	Returning to the problems \eqref{reduced Dirac problem 1} and \eqref{reduced Dirac problem 2}, a very natural idea is to apply the above abstract framework to the functionals given by $L_0=\cj_0$ and $L_\vr=\cj_\vr$ on $\ch=\msd^{\frac12}(\R^m,\mbs(\R^m))$ (see Lemma \ref{expansion J functional} for the functionals associated to Eq. \eqref{reduced Dirac problem 1}, while the functionals associated to Eq. \eqref{reduced Dirac problem 2} are much easier to obtain). As was already shown in \cite[Section 5, 6]{Isobe13} that $\cj_0$ satisfies $(A1)$-$(A3)$ for a critical manifold $\cm$ defined as
	\begin{\equ}\label{critical manifold}
		\cm:=\big\{\psi_{\lm,\xi,\ga}:\, \lm>0,\ \xi\in\R^m,\ \ga\in\mbs_m,\ |\ga|=1\big\},
	\end{\equ}
	where
	\begin{\equ}\label{critical manifold explicit}
		\psi_{\lm,\xi,\ga}(x)=\frac{m^{\frac{m-1}2}\lm^{\frac{m-1}2}}
		{\big(\lm^2+|x-\xi|^2 \big)^{\frac m2}}(\lm-(x-\xi))\cdot_{\ig_{\R^m}}\ga
	\end{\equ}
	for $\lm>0$, $\xi\in\R^m$, $\ga\in\mbs_m$ with $|\ga|=1$ ($\mbs_m$ is the spinor module, see\cite{Friedrich00, LM}) and $\cdot_{\ig_{\R^m}}$ denotes the Clifford multiplication with respect to the Euclidean metric. Note that $\cm$ is diffeomorphic to $(0,\infty)\times\R^m\times S^{2^{[\frac m2]+1}-1}(\mbs_m)$ via the canonical map $(\lm,\xi,\ga)\mapsto \psi_{\lm,\xi,\ga}$, where $S^{2^{[\frac m2]+1}-1}(\mbs_m)$ stands for the $(2^{[\frac m2]+1}-1)$-dimensional unit sphere in $\mbs_m$. And hence $\cm$ is a non-compact manifold and the dimension of $\cm$ is $m+2^{[\frac m2]+1}$.
	\end{Rem}
	
	\subsection{Perturbation method with  degenerate conditions}
	
	Unfortunately, in the spinorial setting, the reduced functional $L_\eps^{red}$ happens to have much worse analytic properties than the usual cases, and one of these ``bad'' behaviors is the degeneracy on $\cm$. This, for instance, can be seen from the explicit formulations of those perturbation terms in \eqref{expansion of cj-vr} where $\Ga$ and $\Phi$ do not depend on all variables of $\cm$ (in fact, if we substitute \eqref{critical manifold}-\eqref{critical manifold explicit} into \eqref{expansion of cj-vr}, we find that $\Ga$ and $\Phi$ do not depend on the variable $\ga$ in $\cm$). Hence, critical points of the functional $\cj_\vr$ can not be obtained via non-degenerate arguments, in particular,	standard methods as in \cite{AM99, AM, Chang, MW} do not apply.
	
	Here we recall a recent framework developed in \cite[Section 2]{Isboe-Xu21}, which can be employed to handle spinor field equations like \eqref{reduced Dirac problem 1} and \eqref{reduced Dirac problem 2}. To see this, besides the assumptions $(A1)$-$(A3)$, we will need the following additional conditions for the critical manifold $\cm$:
	\begin{itemize}
		\item [$(A4)$] $\cm$ admits a (globally) trivializable fiber bundle structure over a compact base space $\cn$ with projection $\vartheta:\cm\to\cn$ and fiber $\cg$. Precisely, there is a fiber preserving diffeomorphism $\iota:\cg\times\cn\to\cm$ such that the following diagram commutes
		\begin{displaymath}
			\xymatrix{
				\cg\times\cn \ar[r]^{\ \ \iota}  \ar[d]_{Proj} & \cm \ar[d]^{\vartheta} \\
				\cn\ar[r]^{\ \id} & \cn
			}
		\end{displaymath}

		\item[$(A5)$]  $T_\ga\,\cn\subset \ker\nabla(\Ga\circ\iota)(g,\ga)$ for any $(g,\ga)\in \cg\times\cn$, where we have identified $T_\ga\,\cn$ as a subspace of the total tangent space $T_{(g,\ga)}(\cg\times\cn)$.
		
	\end{itemize}

	\begin{Rem}
		\begin{itemize}
			\item[(1)] In our application $\cn=S^{2^{[\frac m2]+1}-1}(\mbs_m)$, $\cg=(0,+\infty)\times\R^m$ and $\iota(g,\ga):=\psi_{\lm,\xi,\ga}$ for $g=(\lm,\xi)\in\cg$ and $\ga\in\cn$, hence we have a very natural bundle structure on $\cm$. Particularly, we note that there is a continuous action $\cg\times\cm\to\cm$ such that $\cg$ preserves the fibers of $\cm$ (i.e. if $(\mu,y)\in \cg$ and $\psi_{\lm,\xi,\ga}\in\cm_\ga$ then $\psi_{\lm,\xi,\ga}*(\mu,y)=\psi_{\lm\mu,\,\xi+y,\ga}\in\cm_\ga$). Hence the critical manifold in \eqref{critical manifold} is essentially a principal $\cg$-bundle. And since it admits a global section, we easily see that $\cm$ is trivializable. This is the reason we introduce condition $(A4)$.
			
			\item[(2)] Note that if $\cm$ is parameterized via the map $\iota$, condition $(A4)$ makes the variational problem even clearer: it is equivalent to consider the functional $L_\eps^{red}\circ\iota:\cg\times\cn\to\R$.  Comparing with the standard theory in \cite{AM, Chang, MW}, the distinct new feature $(A5)$ describes a certain degenerate situation and, particularly, it implies that $\Ga\circ\iota(g,\ga)$ depends only on the variables in the fiber space $\cg$. Thus we shall turn to study $\tilde\Ga(g)=\Ga\circ\iota(g,\ga)$. For later use, we distinguish $(A5)$ into the following two cases:
			\[
			\left\{
			\aligned
			&\ker\nabla(\Ga\circ\iota)(g,\ga)\equiv T_{(g,\ga)}(\cg\times\cn) & &\text{for all } (g,\ga)\in\cg\times\cn, \\[0.5em]
			&\ker\nabla(\Ga\circ\iota)(g,\ga)\neq T_{(g,\ga)}(\cg\times\cn) & &\text{for some } (g,\ga)\in\cg\times\cn, 
			\endaligned	\right.
			\]
			and we will collect two abstract results which are useful in the spinorial Yamabe-type problems.
		\end{itemize}
	\end{Rem}

\subsubsection*{Case 1: $\ker\nabla(\Ga\circ\iota)(g,\ga)\equiv T_{(g,\ga)}(\cg\times\cn)$ for all $(g,\ga)\in\cg\times\cn$}

In this setting, we have $\Ga\circ\iota(g,\ga)\equiv constant$ on $\cg\times\cn$ and we need to evaluate further terms in the expansion of $L_\eps^{red}$. For this purpose, let us develop the expansion \eqref{model problem} in powers of $\eps$ as
\begin{\equ}\label{model problem2}
	L_\eps(z)=L_0(z)+\eps\Ga(z)+\eps^2\Phi(z)+o(\eps^2)
\end{\equ}

Note that $\Ga\circ\iota(g,\ga)\equiv constant$ on $\cg\times\cn$ is equivalent to $\Ga(z)\equiv constant$ on $\cm$. It follows that $\nabla \Ga(z)\in \cw_z:=T_z\cm^\bot$. Recall that $w_\eps(z)$ is the solution to the auxiliary equation $P_z\nabla L_\eps(z+w)=0$, hence we have
$\nabla L_\eps(z+w_\eps(z))\in T_z\cm$. For a fixed $z\in\cm$, using Taylor expansion, one sees
\[
\aligned
\nabla L_\eps(z+w_\eps(z))&=\nabla L_0(z+w_\eps(z))+\eps\nabla\Ga(z+w_\eps(z))+o(\eps)\\[0.2em]
&=\nabla^2L_0(z)[w_\eps(z)] +\eps\nabla\Ga(z)+\eps\nabla^2\Ga(z)[w_\eps(z)] + o(\|w_\eps(z)\|) +o(\eps).
\endaligned
\]
Then, form \eqref{AM esti} and the fact $\nabla L_\eps(z+w_\eps(z))\in T_z\cm$, it follows that
\[
\nabla^2L_0(z)[w_\eps(z)]+\eps \nabla\Ga(z)+o(\eps)\in T_z\cm.
\]
And hence, by projecting the above equation into $\cw_z$, we deduce
\begin{\equ}\label{w-vr}
w_\eps(z)=-\eps K_z(\nabla\Ga(z))+o(\eps),
\end{\equ}
where $K_z$ stands for the inverse of $\nabla^2L_0(z)$ restricted to $\cw_z$. Now, we can expand $L_\eps^{red}(z):=L_\eps(z+w_\eps(z))$ as
\begin{\equ}\label{L-vr-expansion-deeper}
\aligned
L_\eps^{red}(z)&=L_0(z)+\frac12\nabla^2L_0(z)[w_\eps(z),w_\eps(z)]\\
&\qquad +\eps\Ga(z)+\eps\nabla\Ga(z)[w_\eps(z)]+\eps^2\Phi(z)+o(\eps^2)\\
&=L_0(z)+\eps \Ga(z)+\eps^2\Big( \Phi(z)-\frac12\inp{K_z(\nabla\Ga(z))}{\nabla\Ga(z)} \Big) +o(\eps^2).
\endaligned
\end{\equ}
Here, we emphasize that both $L_0(z)$ and $\Ga(z)$ are constants on $\cm$. The following result is due to \cite[Theorem 2.6]{Isboe-Xu21}.

\begin{Thm}\label{abstract result3}
	Let $L_0,\Ga,\Phi\in C^2(\ch,\R)$ as in \eqref{model problem2} and suppose that $(A1)$-$(A5)$ are satisfied. If there is an open bounded subset $U\subset\cg$ such that
	\[
	\inf_{\ga\in\cn}\Big(\min_{\pa U}\hat\Phi\big|_{\vartheta^{-1}(\ga)} -\min_{\ov U}\hat\Phi\big|_{\vartheta^{-1}(\ga)}\Big)>0 \quad \text{or} \quad 
	\sup_{\ga\in\cn}\Big(\max_{\pa U}\hat\Phi\big|_{\vartheta^{-1}(\ga)} -\max_{\ov U}\hat\Phi\big|_{\vartheta^{-1}(\ga)}\Big)<0,
	\]
	where $\hat\Phi\big|_{\vartheta^{-1}(\ga)}=\hat\Phi\circ\iota(\cdot,\ga)$ and
	\[
	\hat\Phi(z):=\Phi(z)-\frac12\inp{K_z(\nabla\Ga(z))}{\nabla\Ga(z)} \quad \text{for } z\in\cm.
	\]
	Then, for $|\eps|$ small, the functional $L_\eps$ has a critical point on $\cm$.
\end{Thm}

\subsubsection*{Case 2: $\ker\nabla(\Ga\circ\iota)(g,\ga)\neq T_{(g,\ga)}(\cg\times\cn)$ for some   $(g,\ga)\in\cg\times\cn$}

Clearly, in this case, $\tilde\Ga(g)=\Ga\circ\iota(g,\ga)\neq constant $ on $\cg\times\cn$ (evidently, $\Ga(z)\neq constant$ on $\cm$). And the existence result is as follows, we refer the reads to \cite[Theorem 2.4 and Remark 2.5]{Isboe-Xu21}.
	
\begin{Thm}\label{abstract result2}
	Let $L_0,\Ga\in C^2(\ch,\R)$ as in \eqref{model problem} and suppose that $(A1)$-$(A5)$ are satisfied. If $\tilde\Ga$ is a Morse function on $\cg$ and there is an open bounded subset $\Om\subset\cg$ such that the topological degree $\tdeg(\nabla\tilde\Ga,\Om,0)\neq0$. Then, for $|\eps|$ small, the functional $L_\eps^{red}$ has at least $\cat(\cn)$ critical points on $\cm$. In particular, for each critical point $\bar g$ of $\tilde\Ga$, there exist at least $\cat(\cn)$ critical points $(g(\ga),\ga)\in\cg\times\cn$ of $L_\eps^{red}$, each of which satisfies $g(\ga)=\bar g +o(1)$ as $\eps\to0$.
\end{Thm}
Here $\cat(\cn)$ denotes the Lusternik-Schnierelman category of $\cn$, namely the smallest integer $k$ such that $\cn\subset \cup_{i=1}^k A_k$, where the sets $A_k$ are closed and contractible in $\cn$.	
	
\section{The non-compactness caused by the background metric}\label{sec:blow-up}

In this section, let us consider the Eq. \eqref{reduced Dirac problem 1}, where the background metric is given by \eqref{the metric g} and \eqref{metric form}. In this setting, our main purpose is to use our abstract result to prove Theorem \ref{main thm}. Here we emphasis that the functional $\cj_0$ in Lemma \ref{expansion J functional} plays the role of $L_0$ in our abstract settings.
	
\subsection{Some basic facts}

We first report some important properties of the functionals $\Ga$ and $\Phi$ in Lemma \ref{expansion J functional}, which have been shown in \cite{Isboe-Xu21}. Recall that since the critical manifold $\cm\subset\msd^{\frac12}(\R^m,\mbs(\R^m))$ for $\cj_0$ is given by \eqref{critical manifold}-\eqref{critical manifold explicit}, we have
\begin{Lem}\label{first properties}
Assume that we are in the hypotheses of Lemma \ref{expansion J functional}, 
for $\psi_{\lm,\xi,\ga}\in\cm$ with $\lm>0$, $\xi\in\R^m$ and $\ga\in S^{2^{[\frac m2]+1}-1}(\mbs_m)$, there hold
\[
\Ga(\psi_{\lm,\xi,\ga})\equiv0
\]
and
\[
\Phi(\psi_{\lm,\xi,\ga})=\frac{m^{m-1}\lm^m}{16}\int_{\R^m}\frac{ \tr(\tilde\ih^2)-(\tr\tilde\ih)^2}{\big( \lm^2+|x-\xi|^2 \big)^m} d\vol_{\ig_{\R^m}}.
\]  
Moreover
\begin{itemize}
	\item[$(1)$]
	$\displaystyle 
	\lim_{\lm\to0}	\Phi(\psi_{\lm,\xi,\ga})=C_0\big( \tr(\tilde\ih^2)-(\tr\tilde\ih)^2 \big)(\xi)$ for any $\xi\in\R^m$,
	where
	\[
	C_0=\frac{m^{m-1}}{16}\int_{\R^m}\frac{1}{\big( 1+|x|^2 \big)^m}d\vol_{\ig_{\R^m}};
	\]  
	
	\item[$(2)$] for all $v\in \cw_z:=T_z\cm^\bot$,
		\begin{eqnarray}\label{nabla Ga}
		\inp{\nabla\Ga(z)}{v}
		&=& \frac14\real\int_{\R^m}(\nabla(\tr\tilde\ih)\cdot_{\ig_{\R^m}}z,v)_{\ig_{\R^m}}d\vol_{\ig_{\R^m}} \nonumber \\[0.5em]
		& & -\frac12\sum_i\real\int_{\R^m}\tilde\ih_{ii}(\pa_i\cdot_{\ig_{\R^m}}\nabla_{\pa_i} z,v)_{\ig_{\R^m}}d\vol_{\ig_{\R^m}}  \nonumber \\[0.5em]
		& &
		-\frac14\sum_i\real\int_{\R^m}\pa_i\tilde\ih_{ii}(\pa_i\cdot_{\ig_{\R^m}}z,v)_{\ig_{\R^m}}d\vol_{\ig_{\R^m}};
	\end{eqnarray}
	
	\item[$(3)$] 
	$\displaystyle 
	\lim_{\lm\to0}\inp{K_{\psi_{\lm,\xi,\ga}}(\nabla\Ga(\psi_{\lm,\xi,\ga}))}{\nabla\Ga(\psi_{\lm,\xi,\ga})}=C_1\big( \tr(\tilde\ih^2)-(\tr\tilde\ih)^2 \big)(\xi)$ for any  $\xi\in\R^m$,
	where
	\[
	C_1=\frac{m^{m-1}}{4}\int_{\R^m}\frac{|x|^2}{\big( 1+|x|^2 \big)^{m+1}}d\vol_{\ig_{\R^m}}
	\]  
	and $K_z$ stands for the inverse of $\nabla^2 \cj_0(z)$ restricted to $\cw_z:=T_z\cm^\bot\subset \msd^{\frac12}(\R^m,\mbs(\R^m))$.
\end{itemize}
\end{Lem}

\begin{Rem}\label{C0C1}
Let us point out that the aforementioned two constants $C_0$ and $C_1$ can be computed explicitly as
\[
C_0=\frac{m^{m-1}\om_{m-1}}{16}\int_0^{\infty}\frac{r^{m-1}}{(1+r^2)^m}dr=\frac{m^{m-1}\om_{m-1}}{32}B\Big(\frac m2,\frac m2\Big)
\]
and
\[
C_1=\frac{m^{m-1}\om_{m-1}}{4}\int_0^{\infty}\frac{r^{m+1}}{(1+r^2)^{m+1}}dr=\frac{m^{m-1}\om_{m-1}}{8}B\Big(\frac m2,\frac m2+1\Big)
\]
where $B(x,y)$, defined for $x,y>0$, is the beta function classified by the first kind of Euler's integral. Using the property
\[
B(x,x+1)=\frac12B(x,x), \quad \text{for } x>0
\]
we find $C_0=\frac12 C_1$.
\end{Rem}

Next, let $\cj_\vr^\ih$ be the Euler functional corresponding to the metric $\tilde\ig^{\ih}=\ig_{\R^m}+\vr\ih$, where $\ih$ is a fixed elementary matrix (see Definition \ref{def k-elementary}). Then Lemma \ref{first properties} can be performed also for $\cj_\vr^\ih$. Let $\Ga^\ih$ and $\Phi^\ih$ be the corresponding functionals appearing in the expansion of $\cj_\vr^\ih$. A more detailed characterization for the reorganized functional  
\begin{\equ}\label{the reorganized functional}
\hat\Phi^\ih(z):=\Phi^\ih(z)-\frac12\inp{K_z(\nabla\Ga^\ih(z))}{\nabla\Ga^\ih(z)} \quad \text{for } z\in\cm.
\end{\equ}
can be summarized in the following proposition. We emphasize that, by Lemma \ref{first properties} and Remark \ref{C0C1}, there holds
\[
\lim_{\lm\to0}\hat\Phi^\ih(\psi_{\lm,\xi,\ga})=0
\]
for any $\xi\in\R^m$ and $\ga\in S^{2^{[\frac m2]+1}-1}(\mbs_m)$.

\begin{Prop}\label{prop hat-Phi-ih}
	For $m\geq4$, assume that we are in the hypotheses of Lemma \ref{expansion J functional}. Let $k\in\{1,\dots,m\}$, $p\in[2,\infty]$ and $\ih=\diag(\ih_{11},\dots,\ih_{mm})$ be  $(k,p)$-elementary at a point $\xi\in\R^m$ with $\pa_k\ih_{ii}(\xi)\equiv c_k\neq0$, for $i\neq k$. If
	\[
	\begin{cases}
		p=\infty & m=4,\\
		p>2 & m=5,\\
		p\geq2 & m\geq6,
	\end{cases}
	\]
	then 
	\[
	\hat\Phi^\ih(\psi_{\lm,\xi,\ga})=-\frac{3m^{m-2}(m-1)(m-2)c_k^2}{128} \lm^2\int_{\R^m}\frac{|x|^2}{(1+|x|^2)^m}d\vol_{\ig_{\R^m}}+o(\lm^2) \quad \text{as }\lm\to0.
	\]
	In particular, $\hat\Phi^\ih(\psi_{\lm,\xi,\ga})<0$ for small values of $\lm$. Furthermore,
	\[
	\hat\Phi^\ih(\psi_{\lm,\xi,\ga})\to0 \quad \text{as } \lm+|\xi|\to\infty.
	\]
\end{Prop}

\begin{Rem}
The proof of Proposition \ref{prop hat-Phi-ih} is very technical, we refer to \cite[Section 4.2]{Isboe-Xu21} for more details. We only point out here that the main ingredient lies in characterizing the mapping $w_\vr^{\ih}:\cm\to T\cm^\bot$, $w_\vr^{\ih}(z)=K_z(\nabla\Ga^\ih(z))$, or equivalently solving the equation $\nabla^2\cj_0(z)[w_\vr^\ih(z)]=\nabla\Ga(z)$ for $z\in\cm$. And the $(k,p)$-elementary matrix makes the computation more accessible than using general choices of $\ih$.
\end{Rem}

\medskip

Now, through the perturbation framework introduced in Section \ref{sec:abstract results}, we intend to reduce the problem \eqref{SY-m-sphere} to a finite-dimensional one. Notice that, when $\tilde\ih\in\ch(p)$ (see the definition above Theorem \ref{main thm}), we find the functionals $\Ga$ and $\Phi$ (in the expansion of $\cj_\vr$, see Lemma \ref{expansion J functional}) are actually in the form of summing up infinitely many distinguished terms. For this reason more careful analysis is required. 

\begin{Lem}\label{lem:inequalities}
For $m\geq2$, assume that we are in the hypotheses of Lemma \ref{expansion J functional} with $\tilde\ih\in\ch(p)$, some $p\in(2,\infty)\cup\{\infty\}$. Let $\psi,\va\in\msd^{\frac12}(\R^m,\mbs(\R^m))$ and $z\in\cm$. Then, via the Bourguignon-Gauduchon identification \eqref{spinor identify}, there exists a constant $C>0$ such that the following estimates hold for all $|\vr|$ small:
\begin{\equ}\label{inequ1}
	\cj(\tilde\psi)-\cj_0(\psi)-\vr\Ga(\psi)-\vr^2\Phi(\psi)=o(\vr^2)\big( \|\psi\|^2+\|\psi\|^{\frac{2m}{m-1}} \big);
\end{\equ}
\begin{\equ}\label{inequ2}
	\big\|\nabla\cj_\vr(\tilde\psi)-\nabla\cj_0(\psi)-\vr\nabla\Ga(\psi)\big\|=O(\vr^2)\big( \|\psi\|+\|\psi\|^{\frac{m+1}{m-1}} \big);
\end{\equ}
\begin{\equ}\label{inequ3}
\big\|\nabla \cj_\vr(\tilde z)\big\|=O(|\vr|);
\end{\equ}
\begin{\equ}\label{inequ4}
	\big\|\nabla^2\cj_\vr(\tilde\psi)-\nabla^2\cj_0(\psi)\big\|=O(|\vr|)\big( 1+\|\psi\|^{\frac2{m-1}} \big);
\end{\equ}
\begin{\equ}\label{inequ5}
	|\cj_\vr(\tilde\psi+\tilde\va)-\cj_\vr(\tilde\psi)|\leq C\|\va\|\big( \|\psi\|+\|\va\|+\|\psi\|^{\frac{m+1}{m-1}}+\|\va\|^{\frac{m+1}{m-1}} \big);
\end{\equ}
\begin{\equ}\label{inequ6}
	\big\|\nabla\cj_\vr(\tilde\psi+\tilde\va)-\nabla\cj_\vr(\tilde\psi)\big\|\leq C\|\va\|\big( 1+\|\psi\|^{\frac2{m-1}}+\|\va\|^{\frac2{m-1}} \big);
\end{\equ}
\begin{\equ}\label{inequ7}
	\big\|\nabla\Ga(\tilde\psi+\tilde\va)-\nabla\Ga(\tilde\psi)\big\|\leq C\|\va\|\big( 1+\|\psi\|^{\frac2{m-1}}+\|\va\|^{\frac2{m-1}} \big);
\end{\equ}
\begin{\equ}\label{inequ8}
	\big\|\nabla^2\cj_\vr(\tilde\psi+\tilde\va)-\nabla^2\cj_\vr(\tilde\psi)\big\|\leq \begin{cases}
	 C\|\va\|(\|\psi\|+\|\va\|) & m=2, \\[0.5em]
	C\big(\|\va\|^{\frac2{m-1}}+\|\va\|\big)& m\geq3,
	\end{cases}
\end{\equ}
uniformly in $\psi$, $\va$ and $z$.
\end{Lem}

Without breaking the reading, the proof of Lemma \ref{lem:inequalities} will be given in  Appendix \ref{proof of lemma 4.5}. Now, as an important consequence, we have

\begin{Prop}\label{w-chi}
For $m\geq2$, assume that we are in the hypotheses of Lemma \ref{expansion J functional} with $\tilde\ih\in\ch(p)$, some $p\in(2,\infty)\cup\{\infty\}$. There exists a $C^1$ mapping
\[
(w,\chi):(-\vr_0,\vr_0)\times\cm\to \msd^{\frac12}(\R^m,\mbs(\R^m))\times \msd^{\frac12}(\R^m,\mbs(\R^m))
\]
for some $\vr_0>0$, which satisfies
\begin{itemize}
	\item[$(1)$] $w_\vr(z)=w(\vr,z)\in T_z\cm^\bot$;
	
	\item[$(2)$] $\nabla\cj_\vr\big(\tilde z + \widetilde{w_\vr(z)}\big)=\chi(\vr,z)\in T_z\cm$ for all $z\in\cm$ (via the Bourguignon-Gauduchon identification);
	
	\item[$(3)$] $w_\vr(z)=-\vr K_z(\nabla\Ga(z))+O(|\vr|^\mu)$ with $\mu=2$ for $m=2$ and $\mu=\frac{m+1}{m-1}$ for $\mu\geq3$;
	
	\item[$(4)$] denoted by $\cj_\vr^{red}(z)=\cj_\vr\big(\tilde z + \widetilde{w_\vr(z)}\big)$, then $\tilde z+\widetilde{w_\vr(z)}$ is a critical point of $\cj_\vr$ provided that $z\in\cm$ is a critical point of $\cj_\vr^{red}$.
\end{itemize}
\end{Prop}
\begin{proof}
To obtain the existence of $(w,\chi)$, let us define a mapping $H:\cm\times (-\vr_1,\vr_1)\times \msd^{\frac12}(\R^m,\mbs(\R^m))\times T\cm\to \msd^{\frac12}(\R^m,\mbs(\R^m))\times T\cm$
\[
H(z,\vr,w,\chi)=\big( \nabla\cj_\vr(\tilde z+\tilde w)-\chi,(I-P_z)w \big)
\]
where $(-\vr_1,\vr_1)$ is an interval such that $\cj_\vr$ is well-defined and $P_z:\msd^{\frac12}(\R^m,\mbs(\R^m))\to T_z\cm^\bot$ is the orthogonal projection onto $T_z\cm^\bot$.

Plainly, we have $H(z,0,0,0)=(0,0)$ for all $z\in\cm$. And by elementary computation, we have
\[
\nabla_{(w,\chi)}H(z,0,0,0)[\va,\phi]=\big( \nabla^2\cj_0(z)[\va]-\phi,(I-P_z)\va \big)
\]
for $\va\in \msd^{\frac12}(\R^m,\mbs(\R^m))$ and $\phi\in T_z\cm$. Hence, it follows from the invertibility of $\nabla^2\cj_0(z)|_{T_z\cm^\bot}$ that $\nabla_{(w,\chi)}H(z,0,0,0)$ is invertible and 
\begin{\equ}\label{invert-bdd}
\|\nabla_{(w,\chi)}H(z,0,0,0)^{-1}\|\leq C_0, \quad \forall z\in\cm
\end{\equ}
for some $C_0>0$. Then, by applying the Implicit Function Theorem, one soon obtains the existence of $(w(\vr,z),\chi(\vr,z))$ such that $H(z,\vr,w(\vr,z),\chi(\vr,z))=(0,0)$. This proves $(1)$ and $(2)$.

To see $(3)$ we need more careful analysis of the mapping $w_\vr(z)=w(\vr,z)$. To start with, let us use the invertibility of $\nabla_{(w,\chi)}H(z,0,0,0)$ to define $F_{z,\vr}: \msd^{\frac12}(\R^m,\mbs(\R^m))\times T_z\cm\to \msd^{\frac12}(\R^m,\mbs(\R^m))\times T_z\cm$
\[
F_{z,\vr}(\va,\phi)=-\nabla_{(w,\chi)}H(z,0,0,0)^{-1}\Big( H(z,\vr,\va,\phi)-\nabla_{(w,\chi)}H(z,0,0,0)[\va,\phi] \Big).
\]
Then we can see that $(w(\vr,z),\chi(\vr,z))$ is a fixed point of $F_{z,\vr}$. We claim that 
\begin{claim*}
There exist $L_0,\vr_0>0$ such that, for any given $L>L_0$, $F_{z,\vr}$ is a contraction mapping on $B_{\vr}:=\big\{(\va,\phi)\in \msd^{\frac12}(\R^m,\mbs(\R^m))\times T_z\cm:\, \|\va\|^2+\|\phi\|^2\leq L^2\vr^2\big\}$, for all $\vr\in(-\vr_0,\vr_0)$.
\end{claim*}

We only need to show that $F_{z,\vr}(\va,\phi)\in B_{\vr}$ and
\[
\|F_{z,\vr}(\va_1,\phi_1)-F_{z,\vr}(\va_2,\phi_2)\|\leq \de \|(\va_1,\phi_1)-(\va_1,\phi_1)\|
\]
for all $(\va_1,\phi_1),(\va_2,\phi_2)\in B_\vr$, where $\de\in(0,1)$. And, by \eqref{invert-bdd}, it is enough to show that
\begin{\equ}\label{cmt1}
\big\| \nabla\cj_\vr(\tilde z+\tilde\va)-\nabla^2\cj_0(z)[\va]\big\|\leq \frac{L|\vr|}{C_0}
\end{\equ}
and
\begin{\equ}\label{cmt2}
	\aligned
&\big\| \big(\nabla\cj_\vr(\tilde z+\tilde\va_1)-\nabla^2\cj_0(z)[\va_1]\big)-\big(\nabla\cj_\vr(\tilde z+\tilde\va_2)-\nabla^2\cj_0(z)[\va_2]\big)\big\|
\\
&\qquad \leq \frac{\de}{C_0}\|(\va_1,\phi_1)-(\va_1,\phi_1)\|.
\endaligned
\end{\equ}
Here and in the sequel, we will used the expressions given in Lemma \ref{expansion J functional} so that the gradient map $\nabla\cj_\vr(\cdot)$ is appropriately identified in the space $\msd^{\frac12}(\R^m,\mbs(\R^m))$, and it will cause no confusion if we make the difference between $\nabla\cj_\vr$ and derivatives of $\cj_0$. We shall also adopt such identification for higher order derivatives of $\cj_\vr$.

Using \eqref{inequ3}, \eqref{inequ4} and \eqref{inequ8}, we find
\[
\aligned
&\big\|\nabla\cj_\vr(\tilde z+\tilde\va)-\nabla^2\cj_0(z)[\va]\big\| \\[0.3em]
&\qquad = \big\|\nabla\cj_\vr(\tilde z+\tilde\va)-\nabla\cj_\vr(\tilde z)-\nabla^2\cj_\vr(\tilde z)[\va]+\nabla\cj_\vr(\tilde z)+(\nabla^2\cj_\vr(\tilde z)-\nabla^2\cj_0(z))[\va] \big\|\\[0.3em]
&\qquad \leq\int_0^1\big\| \nabla^2\cj_\vr(\tilde z+s\tilde\va)-\nabla^2\cj_\vr(\tilde z) \big\|\|\va\|ds + O(|\vr|) + O(|\vr|)\|\va\| \\[0.3em]
&\qquad \leq \begin{cases}
O(1)\|\va\|^2+O(1)\|\va\|^3+O(|\vr|)+O(|\vr|)\|\va\| & \text{if } m=2 \\[0.5em]
O(1)\|\va\|^{\frac{m+1}{m-1}}+O(1)\|\va\|^2+O(|\vr|)+O(|\vr|)\|\va\| & \text{if } m\geq3
\end{cases}
\endaligned
\]
since $\|z\|$ is uniformly bounded for $z\in\cm$. This proves \eqref{cmt1} when $L$ is fixed reasonably large. To see \eqref{cmt2}, we point out that
\[
\aligned
&\nabla\cj_\vr(\tilde z+\tilde\va_1)-\nabla\cj_\vr(\tilde z+\tilde\va_2)-\nabla^2\cj_0(z)[\va_1-\va_2] \\[0.3em]
&\qquad =\int_0^1\nabla^2\cj_\vr\big(\tilde z+\tilde\va_2+s(\tilde\va_1-\tilde\va_2)\big)[\tilde\va_1-\tilde\va_2]ds-\nabla^2\cj_0(z)[\va_1-\va_2]
\endaligned
\] 
and hence by \eqref{inequ4} and \eqref{inequ8} we get
\[
\aligned
&\big\| \big(\nabla\cj_\vr(\tilde z+\tilde\va_1)-\nabla^2\cj_0(z)[\va_1]\big)-\big(\nabla\cj_\vr(\tilde z+\tilde\va_2)-\nabla^2\cj_0(z)[\va_2]\big)\big\|
\\[0.3em]
&\qquad \leq O(|\vr|)\max_{s\in[0,1]}\big( 1+ \|z+\va_2+s(\va_1-\va_2)\|^{\frac2{m-1}}\big)\|\va_1-\va_2\| \\
&\qquad + \begin{cases}
\displaystyle O(1)\max_{s\in[0,1]}\|\va_2+s(\va_1-\va_2)\|\big( \|z\|+\|\va_2+s(\va_1-\va_2)\| \big)\|\va_1-\va_2\| & \text{if } m=2, \\[0.6em]
\displaystyle O(1)\max_{s\in[0,1]}\big( \|\va_2+s(\va_1-\va_2)\|^{\frac2{m-1}}+\|\va_2+s(\va_1-\va_2)\| \big)\|\va_1-\va_2\| & \text{if }m\geq3.
\end{cases}
\endaligned
\]
Therefore, when $|\vr|$ is small enough, we obtain \eqref{cmt2}. And the claim is proved. 

\medskip

As an immediate consequence of the above claim, we find that $F_{z,\vr}$ always has a fixed point in $B_\vr$. Hence we conclude that $\|w_\vr(z)\|=\|w(\vr,z)\|\leq L\vr$, with $L>L_0$ being fixed. Now, in order to prove $(3)$, we write
\[
\aligned
\nabla\cj_\vr\big(\tilde z + \widetilde{w_\vr(z)}\big)&=\nabla\cj_\vr\big(\tilde z + \widetilde{w_\vr(z)}\big)-\nabla\cj_0(z+w_\vr(z))-\vr\nabla\Ga(z+w_\vr(z)) \\[0.3em]
&\qquad +\nabla\cj_0(z+w_\vr(z))-\nabla^2\cj_0(z)[w_\vr(z)] \\[0.3em]
&\qquad +\vr\nabla\Ga(z+w_\vr(z))-\vr\nabla\Ga(z)\\[0.3em]
&\qquad +\nabla^2\cj_0(z)[w_\vr(z)]+\vr\nabla\Ga(z)
\endaligned
\]
in which we can use \eqref{inequ2}, \eqref{inequ7} and \eqref{inequ8} to get
\[
\big\|\nabla\cj_\vr\big(\tilde z + \widetilde{w_\vr(z)}\big)-\nabla\cj_0(z+w_\vr(z))-\vr\nabla\Ga(z+w_\vr(z))\big\|=O(\vr^2),
\]
\[
\aligned
\big\|\nabla\cj_0(z+w_\vr(z))-\nabla^2\cj_0(z)[w_\vr(z)]\big\|&\leq\int_0^1\big\|\nabla^2\cj_0(z+sw_\vr(z))-\nabla^2\cj_0(z)\big\|\|w_\vr(z)\|ds \\[0.3em]
&\leq \begin{cases}
O(1)\|w_\vr(z)\|^2=O(\vr^2) & \text{if } m=2\\[0.5em]
O(1) \|w_\vr(z)\|^{\frac{m+1}{m-1}}=O\big(|\vr|^{\frac{m+1}{m-1}}\big) & \text{if } m\geq3
\end{cases}
\endaligned
\]
and
\[
\big\|\nabla\Ga(z+w_\vr(z))-\nabla\Ga(z)\big\|=O(1)\|w_\vr(z)\|=O(|\vr|).
\]
Hence, we deduce from $(2)$ that
\[
\chi(\vr,z)=\nabla^2\cj_0(z)[w_\vr(z)]+\vr\nabla\Ga(z)+\begin{cases}
	O(1)\|w_\vr(z)\|^2=O(\vr^2) & \text{if } m=2,\\[0.5em]
	O(1) \|w_\vr(z)\|^{\frac{m+1}{m-1}}=O\big(|\vr|^{\frac{m+1}{m-1}}\big) & \text{if } m\geq3.
\end{cases}
\]
Projecting this equation on $T_z\cm^\bot$ and applying the operator $K_z=\big(\nabla^2\cj_0(z)|_{T_z\cm^\bot}\big)^{-1}$ on both sides, we obtain assertion $(3)$. 

Finally $(4)$ is a direct consequence of $(1)$ and $(2)$. This completes the proof.
\end{proof}

\begin{Rem}
Comparing with \eqref{AM esti} and \eqref{w-vr}, though Proposition \ref{w-chi} is quite similar to the framework in Section \ref{sec:abstract results}, we carry out the details here mainly because the functional $\cj_\vr$ is involved with the infinite series $\tilde\ih\in\ch(p)$. Clearly, Proposition \ref{w-chi} suggests that Theorem \ref{abstract result3} can be applied to the functional $\cj_\vr$.
\end{Rem}
	
\subsection{Proof of Theorem \ref{main thm}}

In virtue of Lemma \ref{first properties} and Proposition \ref{prop hat-Phi-ih}, let us denote $\cj_\vr^{(i)}$ the Euler functional corresponding to the metric $\ig^{(i)}_\vr=\ig_{\R^m}+\vr a_i h(x-x_i)$, where $h$ is a given compactly supported elementary matrix satisfying the hypotheses of Proposition \ref{prop hat-Phi-ih}, $a_i\in\R$ and $x_i\in\R^m$ are fixed. Let $\Ga^{(i)}$ and $\Phi^{(i)}$ be the corresponding functionals appearing in the expansion of $\cj_\vr^{(i)}$. According to the abstract setting in Section \ref{sec:abstract results}, we denote $w_\vr(z)$ and $w_\vr^{(i)}(z)$ the solutions to the auxiliary equations $P_z\nabla \cj_\vr(z+w)=0$ and $P_z\nabla\cj_\vr^{(i)}(z+w)=0$, respectively, where $P_z:\msd^{\frac12}(\R^m,\mbs(\R^m))\to T_z\cm^\bot$ is the orthogonal projection.

It can be seen from Proposition \ref{prop hat-Phi-ih} that the reorganized functional $\hat\Phi^\ih$, which is given by \eqref{the reorganized functional}, possesses some negative minimum and tends to zero at the boundary of $(0,+\infty)\times\R^m\times S^{2^{[\frac m2]+1}-1}(\mbs_m)$. Hence, we can find an open bounded subset $U\subset \cg=(0,+\infty)\times\R^m$ and $\de>0$ such that $\ov U\subset \cg$ and
\[
\inf_{\ga\in\cn}\Big(\min_{\pa U}\hat\Phi^\ih(\psi_{\,\cdot,\,\cdot,\ga}) -\min_{\ov U}\hat\Phi^\ih(\psi_{\,\cdot,\,\cdot,\ga})\Big)\geq\de
\]
where $\cn=S^{2^{[\frac m2]+1}-1}(\mbs_m)$. In what follows, we keep this precompact set $U$ being fixed and denote
\[
U_{x_i}=\big\{(\lm,\xi)\in\cg:\, (\lm,\xi-x_i)\in U\big\}.
\]

\begin{Lem}\label{Lem:sum}
\begin{itemize}
	\item[$(1)$] Let $\cm_c$ be a compact subset of $\cm$, then there exists $C>0$ such that for $|\vr|$ small there hold
	\[
	\big\|w_\vr(z)-w_\vr^{(i)}(z)\big\|\leq C|\vr| \big\|\nabla\Ga(z)-\nabla\Ga^{(i)}(z)\big\|
	\]
	for all $z\in\cm_c$\,.
	
	\item[$(2)$] For $z=\psi_{\lm,\xi,\ga}\in\cm$ with $(\lm,\xi)\in \ov U_{x_i}$, there exist $C,L>0$ such that if $|x_j-x_i|\geq L$ for all $j\neq i$ then
	\[
	\big\|\nabla\Ga(z)-\nabla\Ga^{(i)}(z)\big\|\leq C\sum_{\substack{j\geq1 \\ j\neq i}}\frac{|a_j|}{|x_j-x_i|^{m-1}}.
	\]
\end{itemize}
\end{Lem}
\begin{proof}
Since the linear operator $K_z=\nabla^2\cj_0(z)^{-1}:T_z\cm^\bot\to T_z\cm^\bot$ is uniformly bounded for $z\in\cm$ (see \cite[Lemma 4.11]{Isboe-Xu21}), we soon obtain from \eqref{w-vr} and Proposition \ref{w-chi} that
\[
\big\|w_\vr(z)-w_\vr^{(i)}(z)\big\|=|\vr|\big\| K_z(\nabla\Ga(z))-K_z(\nabla\Ga^{(i)}(z)) \big\|+o(\vr),
\]
which proves the assertion $(1)$.

To check $(2)$, let us use Lemma \ref{first properties} $(2)$ (which can be also applied to compute $\nabla\Ga^{(i)}$) to get the estimate: for any $v\in \msd^{1/2}(\R^m,\mbs(\R^m))$
\[
\aligned
&\big|\inp{\nabla\Ga(\psi_{\lm,\xi,\ga})}{v}-\inp{\nabla\Ga^{(i)}(\psi_{\lm,\xi,\ga})}{v}\big| \\[0.3em]
&\qquad \leq C \Big(\sum_{j\neq i} |a_j| \int_{\Om_j}|\nabla\psi_{\lm,\xi,\ga} |\cdot|v| +|\psi_{\lm,\xi,\ga}|\cdot|v|d\vol_{\ig_{\R^m}} \Big) \\
&\qquad \leq C\Big(\sum_{j\neq i}|a_j| \int_{\Om_j}\frac{\lm^{\frac{m-1}2}}{\big(\lm^2+|x-\xi|^2\big)^{\frac m2}}\cdot|v| +\frac{\lm^{\frac{m-1}2}}{\big(\lm^2+|x-\xi|^2\big)^{\frac{m-1}2}}\cdot|v|d\vol_{\ig_{\R^m}} \Big),
\endaligned
\]
where $\Om_j=\supp \ih(\cdot-x_j)$ for $j\geq 1$, and in the last inequality we have used the facts
\[
|\psi_{\lm,\xi,\ga}(x)|_{\ig_{\R^m}}\sim\frac{\lm^{\frac{m-1}{2}}}{\big( \lm^2+|x-\xi|^2 \big)^{\frac{m-1}{2}}} \quad \text{and} \quad 
|\nabla\psi_{\lm,\xi,\ga}(x)|\sim \frac{\lm^{\frac{m-1}{2}}}{\big( \lm^2+|x-\xi|^2 \big)^{\frac{m}{2}}}.
\]
Now, using the H\"older and Sobolev inequalities, we know that for $(\lm,\xi)\in\ov U_{x_i}$ there holds
\[
\big|\inp{\nabla\Ga(\psi_{\lm,\xi,\ga})}{v}-\inp{\nabla\Ga^{(i)}(\psi_{\lm,\xi,\ga})}{v}\big|  \leq C\|v\|\sum_{j\neq i}\Big( \frac{|a_j|}{|x_j-x_i|^{m}} +  \frac{|a_j|}{|x_j-x_i|^{m-1}} \Big)
\]
provided $|x_j-x_i|\geq L$, $j\neq i$, with $L$ large enough (say $L\geq\diam U+1$). This completes the proof.
\end{proof}

The next result will be devoted to compare the values of $\cj_\vr^{red}(z):=\cj_\vr(z+w_\vr(z))$ and $\cj_\vr^{(i),red}(z):=\cj_\vr^{(i)}(z+w_\vr^{i}(z))$ for $z=\psi_{\lm,\xi,\ga}\in\cm$ with $(\lm,\xi)\in \ov U_{x_i}$.

\begin{Prop}\label{prop compare}
There exists $C>0$ such that for $|\vr|$ small there holds
\[
\big|\cj_\vr^{red}(z)-\cj_\vr^{(i),red}(z)  \big|
\leq C|\vr|\bigg( \sum_{\substack{j\geq1 \\ j\neq i}}\frac1{|x_j-x_i|^{\frac{(m-1)\tau}{\tau-1}}} \bigg)^{\frac{\tau-1}{\tau}}
\]
for all $z=\psi_{\lm,\xi,\ga}\in\cm$ with $(\lm,\xi)\in \ov U_{x_i}$, and $|x_j-x_i|\geq L$.
\end{Prop}
\begin{proof}
	Following from Lemma \ref{Lem:sum} and the boundedness of $\nabla \cj_\vr(z+w_\vr(z))$, we have
	\[
	\aligned
	&\big|\cj_\vr^{red}(z)-\cj_\vr^{(i),red}(z)  \big|\\[0.3em]
	&\qquad \leq \big| \cj_\vr(z+w_\vr(z))-\cj_\vr(z+w_\vr^{(i)}(z))  \big| + \big| \cj_\vr(z+w_\vr^{(i)}(z))-\cj_\vr^{(i)}(z+w_\vr^{(i)}(z)) \big|\\[0.3em]
	&\qquad \leq C\big\|w_\vr(z)-w_\vr^{(i)}(z)\big\| + \big| \cj_\vr(z+w_\vr^{(i)}(z))-\cj_\vr^{(i)}(z+w_\vr^{(i)}(z)) \big|\\[0.3em]
	&\qquad\leq |\vr| C\big\|\nabla\Ga(z)-\nabla\Ga^{(i)}(z)\big\| + \big| \cj_\vr(z+w_\vr^{(i)}(z))-\cj_\vr^{(i)}(z+w_\vr^{(i)}(z)) \big|\\[0.3em]
	&\qquad \leq |\vr| C\sum_{\substack{j\geq1 \\ j\neq i}}\frac{|a_j|}{|x_j-x_i|^{m-1}}+ \big| \cj_\vr(z+w_\vr^{(i)}(z))-\cj_\vr^{(i)}(z+w_\vr^{(i)}(z)) \big|.
	\endaligned
	\]
	By noting that
	\[
	\cj_\vr(z+w_\vr^{(i)}(z))-\cj_\vr^{(i)}(z+w_\vr^{(i)}(z))=\vr\Ga(z+w_\vr^{(i)}(z))-\vr\Ga^{(i)}(z+w_\vr^{(i)}(z))+o(\vr)
	\]
	and $\Ga|_\cm=\Ga^{(i)}|_\cm\equiv0$, by using \eqref{AM esti} and Lemma \ref{lem:inequalities}, we have 
	\[
	\Ga(z+w_\vr^{(i)}(z))=\nabla\Ga(z)[w_\vr^{(i)}(z)]+o(\vr)=O(\vr)
	\]
	and
	\[
	\Ga^{(i)}(z+w_\vr^{(i)}(z))=\nabla\Ga^{(i)}(z)[w_\vr^{(i)}(z)]+o(\vr)=O(\vr)
	\]
	for $z=\psi_{\lm,\xi,\ga}\in\cm$ with $(\lm,\xi)\in \ov U_{x_i}$. Then, as long as $|x_j-x_i|\geq L$ for all $j\neq i$, we deduce
	\[
	\big|\cj_\vr^{red}(z)-\cj_\vr^{(i),red}(z)  \big|\leq |\vr| C\sum_{\substack{j\geq1 \\ j\neq i}}\frac{|a_j|}{|x_j-x_i|^{m-1}}
	\leq |\vr|C \bigg(\sum_{\substack{j\geq 1 \\ j\neq i}} \frac1{|x_j-x_i|^{\frac{(m-1)\tau}{\tau-1}}} \bigg)^{\frac{\tau-1}{\tau}}
	\]
	provided that
	\[
	\sum_{j\geq 1}|a_j|^\tau<+\infty.
	\]
\end{proof}

Now we are ready to prove our main results:

\begin{proof}[Complete proof of Theorem \ref{main thm}] For a given $k\geq1$, let us fix arbitrarily $x_0\in\R^m$ with $|x_0|=1$ and take $\tilde\ih$ to be of the form \eqref{h-expansion} with $\ih$ being a compactly supported elementary matrix and satisfying the hypotheses of Proposition \ref{prop hat-Phi-ih}, $a_j=j^{-\bt}$ and $x_j=j^{\al}r x_0$ for $j\in\N$, where 
\begin{\equ}\label{the choice}
r=\frac{C_0}{|\vr|^{1/{m-1}}}, \quad \al>1, \quad 2\al k<\bt
\end{\equ} 
and $C_0>0$ is a constant fixed large enough (see below). With the above choice of $a_j$, we have $\sum_j|a_j|^\tau<+\infty$ since $\bt>1>\frac1\tau$. Note also that $\al>1$, we have $|x_j-x_i|>4\diam(\supp\ih)$ for all $i\neq j$ when $|\vr|$ is small enough. 

From the expansions in Lemma \ref{expansion J functional} and \eqref{L-vr-expansion-deeper}, we have that, for a fixed $i\geq1$,
\[
\cj_\vr^{(i),red}(z)=\cj_0(z)+\vr^2a_i^2\hat\Phi^{\ih(\cdot-x_i)}(z)+o(\vr^2a_i^2)
\]
and, by Proposition \ref{prop hat-Phi-ih}, $\cj_\vr^{(i),red}$ attains a local minimum $z_i=\psi_{\lm_i,\xi_i,\ga_i}$ with $(\lm_i,\xi_i)\in U_{x_i}$ and $\ga_i\in\cn$. In particular, 
\[
\inf_{\ga\in\cn}\Big(\min_{\pa U_{x_i}}\hat\Phi^{\ih(\cdot-x_i)}(\psi_{\,\cdot,\,\cdot,\ga}) -\min_{\ov U_{x_i}}\hat\Phi^{\ih(\cdot-x_i)}(\psi_{\,\cdot,\,\cdot,\ga})\Big)\geq\de
\]
for some $\de>0$ independent of $\vr$ and $i$.

If we choose $C_0$ in \eqref{the choice} so large that $\min_{j\neq i}|x_j-x_i|\geq L$, so Proposition \ref{prop compare} holds, then we have
\[
\aligned
\big|\cj_\vr^{red}(z)-\cj_\vr^{(i),red}(z)  \big|
&\leq C|\vr|\bigg( \sum_{\substack{j\geq1 \\ j\neq i}}\frac1{|x_j-x_i|^{\frac{(m-1)\tau}{\tau-1}}} \bigg)^{\frac{\tau-1}{\tau}} \\
&=\frac{C|\vr|}{r^{m-1}}\bigg( \sum_{\substack{j\geq1 \\ j\neq i}}\frac1{|j^\al-i^\al|^{\frac{(m-1)\tau}{\tau-1}}} \bigg)^{\frac{\tau-1}{\tau}} \\
&\leq \frac{C|\vr|}{r^{m-1}}\cdot\frac{1}{i^{(\al-1)(m-1)}}
\endaligned
\]
when $i$ is fixed large,
where we have used the inequality (see for instance \cite[Lemma 4.4]{BM})
\[
\sum_{\substack{j\geq1 \\ j\neq i}}\frac1{|j^\al-i^\al|^{\frac{(m-1)\tau}{\tau-1}}}\leq \frac{C}{i^{\frac{(\al-1)(m-1)\tau}{\tau-1}}}
\]
for $\al>1$ and $m\geq2$. By enlarging $C_0$ if necessary, we shall have
\[
\frac{C|\vr|}{r^{m-1}}\leq \frac\de4 \vr^2.
\]
And hence, when
\begin{\equ}\label{al-bt-inequ1}
(\al-1)(m-1)\geq2\bt,
\end{\equ}
we find $\cj_\vr^{red}$ has a strict local minimum $\tilde z_i=\psi_{\tilde\lm_1,\tilde\xi_i,\tilde\ga_i}$ ``near" $z_i$ in the sense that  $(\tilde\lm_i,\tilde\xi_i)\in U_{x_i}$.

Summing up, we have proved that if \eqref{al-bt-inequ1} holds then, for all $i$ large and $|\vr|$ small, the functional $\cj_\vr^{red}$ attains a strict local minimum in $U_{x_i}\times\cn$. Hence there are infinitely many distinct solutions of Eq. \eqref{reduced Dirac problem 1}, denoted by $\{\tilde\va_\vr^{(i)}\}$.

To determine the $C^k$-regularity of the metrics $\tilde\ig_\vr$ at infinity and the pull-back metrics $\ig_\vr$ on $S^m$ using our choice \eqref{the choice}, let us denote $\tilde\ig_\vr^\sharp(x)=\tilde\ig_\vr( x/{|x|^2})$ and $\tilde\ig_\vr^{(i),\sharp}(x)=\tilde\ig_\vr^{(i)}(x/|x|^2)$, for $i\in\N$. Since $\ig_\vr$ is smooth on $S^m\setminus\{P\}$, we find that the regularity of $\ig_\vr$ at $P$ is the same of $\tilde\ig_\vr$ at infinity, and so it is the same of $\tilde\ig_\vr^\sharp$ at $0$. If we set $\Om_i=\supp \ih(\cdot-x_i)$, it follows that $\tilde\ig_\vr^{(i),\sharp}(x)-\ig_{\R^m}$ has support $\Om_i^\sharp:=\big\{x\in\R^m:\, x/|x|^2\in\Om_i  \big\}$. In this setting, since $\diam(\Om_i^\sharp)\sim |x_i|^{-2}$, we have the following basic estimate
\[
\|\tilde\ig_\vr^{(i),\sharp}-\ig_{\R^m}\|_{C^k}\leq C|\vr a_i||x_i|^{2k}\leq C\cdot C_0^{2k}|\vr|^{1-2k/(m-1)}i^{2k\al-\bt}
\]
for $k\geq1$. Let 
\[
\tilde\ig_{\vr,j}^\sharp=\ig_{\R^m}+\sum_{i=1}^j \big(\tilde\ig_\vr^{(i),\sharp}(x)-\ig_{\R^m}\big)
\]
we find that if $2k\al-\bt<0$ then
\[
\|\tilde\ig_{\vr,j}^\sharp-\tilde\ig_{\vr,l}^\sharp\|_{C^k}\leq \max_{i=j+1,\dots,l}\|\tilde\ig_\vr^{(i),\sharp}-\ig_{\R^m}\|_{C^k}\leq C\cdot C_0^{2k}|\vr|^{1-2k/(m-1)}(j+1)^{2k\al-\bt}
\]
for all $j<l$. And thus $\{\tilde\ig_{\vr,j}^\sharp\}_{j=1}^{\infty}$ is a Cauchy sequence in $C^k(B_1)$, where $B_1$ stands for the open ball of radius $1$ centered at the origin. Therefore,  $\tilde\ig_\vr^\sharp$ can be extended to $x=0$ in the class of $C^k$. And if there holds additionally that $1-2k/(m-1)>0$, then $\tilde\ig_\vr^\sharp\to\ig_{\R^m}$ as $\vr\to0$. 

There are three essential inequalities in the above arguments, namely \eqref{al-bt-inequ1},
\[
\bt>2k\al \quad \text{and} \quad 2k<m-1.
\]
They are satisfied with $m\geq 4k+2$,
\[
\al>4k+1 \quad \text{and} \quad 2k\al<\bt<2k\al+\frac{\al-(4k+1)}2.
\]

Finally, since the solutions $\{\tilde\va_\vr^{(i)}\}$ of Eq. \eqref{reduced Dirac problem 1} can be parameterized in the compact set $\ov U_{x_i}\times\cn$, via the conformal transformation mentioned in Section \ref{subsec:conformal}, we find that the  corresponding solutions $\{\psi_\vr^{(i)}\}$ of Eq. \eqref{SY-m-sphere} blow up at $P\in S^m$ in the following sense
\[
\|\psi_\vr^{(i)}\|_{L^{\infty}}\to+\infty \quad \text{as } i\to+\infty.
\]
And standard regularity arguments, see \cite[Appendix]{Isobe11} and \cite[Chapter 3]{Ammann}, imply that the weak solutions $\psi_\vr^{(i)}$ are indeed of class $C^1$ on $S^m$. This completes the proof.
\end{proof}

\section{The non-compactness caused by the geometric potential}\label{sec:blow-up2}

In this section, we intend to prove Theorem \ref{main thm1}, and our proof will be based upon the abstract result Theorem \ref{abstract result2} in Section \ref{sec:abstract results}.

Let $H\in C^2(S^m)$ be a given Morse function, satisfying the conditions (H-1) and (H-2) mentioned in Subsection \ref{subsec:conformal}. For simplicity, let us assume $H\geq0$ and takes its minimum at $p_0\in S^m$ and $H(p_0)=0$. Denote $\pi_{p_0}:S^m\setminus\{p_0\}\to\R^m$ the stereographic projection, we define $K(x)=H(\pi_{p_0}^{-1}(x))$ for $x\in\R^m$. Then $K\in L^\infty(\R^m)\cap C^2(\R^m)$ and satisfies the following (see \cite[Lemma 3.1]{Isboe-Xu21})
\begin{\equ}\label{K-inequ-1}
	|\nabla K(x)|\leq C_0(1+|x|^2)^{-1} \quad \text{and} \quad |\nabla^2 K(x)|\leq C_0(1+|x|^2)^{-3/2}
\end{\equ}  
for some constant $C_0>0$.
Taking into account the additional condition $H(p_0)=0$, we also have
\begin{Lem}\label{K-inequ-2}
	By suitably enlarging the constant $C_0$ in \eqref{K-inequ-1} (if necessary), there holds
	\[
	|K(x)|\leq C_0(1+|x|^2)^{-1}
	\]
	for $x\in\R^m$.
\end{Lem}
\begin{proof}
	Since $H(p_0)=0$ and $\nabla_{\ig_{S^m}}H(p_0)=0$, we have
	\[
	|K(x)|=|H(\pi_{p_0}^{-1}(x))-H(p_0)|\leq C|\pi_{p_0}^{-1}(x)-p_0|^2
	\]
	by the Taylor's formula, for some constant $C>0$. By rotation, we may assume that $p_0=(0,0,\dots,1)$ is the north pole and $\pi_{p_0}^{-1}(x)=\big( \frac{2x}{1+|x|^2},\frac{|x|^2-1}{1+|x|^2} \big)$, for $x\in\R^m$. Then the assertion follows from a simple calculation.
\end{proof}

\begin{Rem}
	Though the function $K$ is non-negative in this context, we kept the absolute value symbol in Lemma \ref{K-inequ-2} to emphasis that the inequality also holds true for sign-changing functions and the proof only needs the facts $H(p_0)=0$ and $\nabla_{\ig_{S^m}} H(p_0)=0$.
\end{Rem}

To proceed, let $\{z_i\}_{i=0}^\infty\subset\R^m$ and $\{a_i\}_{i=0}^\infty\subset\R$ be such that
\begin{itemize}
	\item[(1)] $|z_i-z_j|\gg1$ for $i\neq j$. For reader's convenience, we may simply take $z_i=i^\al z_0$ with $z_0\in\R^m$, $|z_0|=R\gg1$ and $\al>1$.
	
	\item[(2)] $a_i=i^{-\bt}$ with $\bt>1$.
\end{itemize}
From $K$, $\{z_i\}$ and $\{a_i\}$ as above, we define
\[
\tilde K(x)=\sum_{i=1}^\infty a_i K(x-z_i).
\]
Then it follows that the above summation converges uniformly in $x$ so that $\tilde K$ is well-defined and $\tilde K\in L^\infty(\R^m)\cap C^2(\R^m)$. In the sequel, for $\psi_{\lm,\xi,\ga}\in\cm$ (see \eqref{critical manifold}-\eqref{critical manifold explicit}), let us set
\[
\tilde\Ga(\psi_{\lm,\xi,\ga})=-\frac{m-1}{2m}\int_{\R^m}\tilde K(x)|\psi_{\lm,\xi,\ga}|_{\ig_{\R^m}}^{\frac{2m}{m-1}}d\vol_{\ig_{\R^m}}
\]
And it is clear that $\tilde\Ga(\psi_{\lm,\xi,\ga})$ is independent of the factor $\ga\in S^{2^{[\frac m2]+1}-1}(\mbs_m)$. Hence, in order to study $\tilde\Ga(\psi_{\lm,\xi,\ga})$, it is sufficient to consider (up to multiplication by a constant)
\[
\tilde\Psi(\lm,\xi):=\int_{\R^m}\tilde K(x)|\psi_{\lm,\xi,\ga}|_{\ig_{\R^m}}^{\frac{2m}{m-1}}d\vol_{\ig_{\R^m}}.
\]
We will also denote $\Ga^{(i)}$, $\Psi^{(i)}$ etc for functions corresponding to $K_i(x)=a_iK(x-z_i)$. Then we have
\[
\tilde\Psi(\lm,\xi)=\sum_{i=1}^\infty \Psi^{(i)}(\lm,\xi)
\]
and
\[
\Psi^{(i)}(\lm,\xi)=m^ma_i\int_{\R^m}\frac{K(\lm x+\xi-z_i)}{(1+|x|^2)^m}d\vol_{\ig_{\R^m}}
\]
where the above formulation comes from a change of variables. The following result is a direct consequence of the computations in \cite[Subsection 3.1]{Isboe-Xu21}, which characterizes the critical points of each $\Psi^{(i)}$ (hence $\Ga^{(i)}$). 

\begin{Prop}\label{prop:topology-degree}
Let $H\in C^2(S^m)$ and $K=H\circ\pi_{p_0}^{-1}$ be as above. Then the critical points of the function $\Psi:\cg=(0,+\infty)\times\R^m\to\R$,
\[
\Psi(\lm,\xi):=m^m\int_{\R^m}\frac{K(\lm x+\xi)}{(1+|x|^2)^m}d\vol_{\ig_{\R^m}},
\]
are isolated and there exists a bounded domain  $\Om_H\subset\cg$ such that
\[
\ov\Om_H\subset \cg, \quad Crit[\Psi]\subset\Om_H \quad \text{and} \quad
\deg(\nabla\Psi,\Om_H,0)\neq0,
\]
where the closure of $\Om_H$ is taken with respect to the standard Euclidean norm and ``$\deg$" stands for the topological degree.
\end{Prop}

As a direct consequence of Proposition \ref{prop:topology-degree}, we can find a bounded domain $\Om_H$ such that all critical points of $\Psi^{(i)}$ are contained in $\Om_H(z_i)=\big\{(\lm,\xi):\, (\lm,\xi-z_i)\in\Om_H \big\}$. Moreover, when $z_i$ and $z_j$ are located far apart, we have $\Om_H(z_i)\cap\Om_H(z_j)=\emptyset$ provided $i\neq j$.

\medskip

Next, we intend to apply Theorem \ref{abstract result2} to prove our second non-compactness result, i.e., Theorem \ref{main thm-H}. The main ingredient here is to show that $\tilde\Psi$ (or equivalently $\tilde\Ga$) has at least one critical point in each $\Om_H(z_i)$. By abuse of notation, we continue to use $\cj_\vr$ for the Euler functional associated to the perturbed problem \eqref{reduced Dirac problem 2}, that is,
\[
\cj_\vr(\psi)=\cj_0(\psi)+\vr\tilde\Ga(\psi)
\]
where $\cj_0$ is as in \eqref{unperturbed functional} and 
\[
\tilde\Ga(\psi)=-\frac{m-1}{2m}\int_{\R^m}\tilde K(x)|\psi|_{\ig_{\R^m}}^{\frac{2m}{m-1}}d\vol_{\ig_{\R^m}}.
\]
We also denote $\cj_\vr^{(i)}$ the functional 
\[
\cj_\vr^{(i)}(\psi)=\cj_0(\psi)+\vr \Ga^{(i)}(\psi)
\]
with
\[
\Ga^{(i)}(\psi)=-\frac{m-1}{2m}\int_{\R^m} K_i(x)|\psi|_{\ig_{\R^m}}^{\frac{2m}{m-1}}d\vol_{\ig_{\R^m}}.
\]
Following from the abstract settings in Section \ref{sec:abstract results}, we will introduce the notation $w_\vr(z)$ and $w_\vr^{i}(z)$ for the solutions to the auxiliary equations $P_z\nabla\cj_\vr(z+w)=0$ and $P_z\nabla \cj_\vr^{(i)}(z+w)=0$, respectively, where $P_z:\msd^{\frac12}(\R^m,\mbs(\R^m))\to T_z\cm^\bot$ stands for the orthogonal projection. Then, analogous to Lemma \ref{Lem:sum}, we have in the present case
\begin{Lem}\label{Lem:sum2}
\begin{itemize}
	\item[$(1)$] Let $\cm_c$ be a compact subset of $\cm$, then there exists $C>0$ such that for $|\vr|$ small there hold
	\[
	\big\|w_\vr(z)-w_\vr^{i}(z)  \big\|\leq C |\vr|\big\|\nabla\tilde\Ga(z)-\nabla\Ga^{(i)}(z) \big\|,
	\]
	for all $z\in\cm_c$\,.
	
	\item[$(2)$] For $z=\psi_{\lm,\xi,\ga}\in\cm$ with $(\lm,\xi)\in\Om_H(z_i)$, there exists $C,L>0$ such that if $|z_j-z_i|\geq L$ for all $j\neq i$ then
	\[
	\big\|\nabla\tilde\Ga(z)-\nabla\Ga^{(i)}(z) \big\|\leq C \sum_{\substack{j\geq 1 \\ j\neq i}} \frac{ |a_j|}{|z_j-z_i|^2}.
	\]
\end{itemize}
\end{Lem}
\begin{proof}
Recall that $w_\vr(z)$ satisfies $P_z\nabla\cj_\vr(z+w_\vr(z))=0$, via Taylor expansion, we find
\[
\aligned
\nabla\cj_\vr(z+w_\vr(z))&=\nabla\cj_0(z+w_\vr(z))+\vr\nabla\tilde\Ga(z+w_\vr(z)) \\[0.3em]
&=\nabla\cj_0(z)+\nabla^2\cj_0(z)[w_\vr(z)]+\vr\nabla\tilde\Ga(z)+\vr\nabla^2\tilde\Ga(z)[w_\vr(z)]+o(\|w_\vr(z)\|).
\endaligned
\]
Since $\nabla\cj_0(z)=0$ for all $z\in\cm$, we get
\[
\nabla\cj_\vr(z+w_\vr(z))=\nabla^2\cj_0(z)[w_\vr(z)]+\vr\nabla\tilde\Ga(z)+\vr\nabla^2\tilde\Ga(z)[w_\vr(z)]+o(\|w_\vr(z)\|),
\]
and the equation $P_z\nabla\cj_\vr(z+w_\vr(z))=0$ becomes
\[
P_z\nabla^2\cj_0(z)[w_\vr(z)]+\vr P_z\nabla\tilde\Ga(z)+\vr P_z\nabla^2\tilde\Ga(z)[w_\vr(z)]+o(\|w_\vr(z)\|)=0.
\]
And a similar equation holds for $w_\vr^{i}(z)$.
Notice that $\cm$ is a non-degenerate critical manifold of $\cj_0$ and $\nabla^2\cj_0(z)$ is invertible on $T_z\cm^\bot$, we find $(1)$ holds true.

\medskip

To check $(2)$, let us take arbitrarily $\va\in\msd^{\frac12}(\R^m,\mbs(\R^m))$ with $\|\va\|\leq1$. Then we have
\[
\aligned
&\Big|\inp{\nabla\tilde\Ga(\psi_{\lm,\xi,\ga})}{\va}-\inp{\nabla\Ga^{(i)}(\psi_{\lm,\xi,\ga})}{\va} \Big|\leq C \sum_{\substack{j\geq 1 \\ j\neq i}}|a_j|\int_{\R^m}|\psi_{\lm,\xi,\ga}|_{\ig_{\R^m}}^{2^*-1}|K(x-z_j)||\va|d\vol_{\ig_{\R^m}}\\
&\qquad \leq C \sum_{\substack{j\geq 1 \\ j\neq i}}|a_j|\int_{\R^m}\frac{\lm^{\frac{m+1}{2}}}{\big(\lm^2+|x-\xi|^2\big)^{\frac{m+1}{2}}}|K(x-z_j)||\va|d\vol_{\ig_{\R^m}}\\
&\qquad \leq C\sum_{\substack{j\geq 1 \\ j\neq i}}|a_j|\int_{\R^m}\frac{\lm^{\frac{m+1}{2}}}{\big(\lm^2+|x-(\xi-z_j)|^2\big)^{\frac{m+1}{2}}}|K(x)||\va(x+z_j)|d\vol_{\ig_{\R^m}}\\
&\qquad \leq C\sum_{\substack{j\geq 1 \\ j\neq i}}|a_j|\int_{\R^m}\frac{\lm^{\frac{m+1}{2}}}{\big(\lm^2+|x-(\xi-z_j)|^2\big)^{\frac{m+1}{2}}}\cdot\frac{1}{1+|x|^2}\cdot|\va(x+z_j)|d\vol_{\ig_{\R^m}}
\endaligned
\]
where the last inequality follows from Lemma \ref{K-inequ-2}. To proceed, let us define
\[
I_j=\int_{\R^m}\frac{\lm^{\frac{m+1}{2}}}{\big(\lm^2+|x-(\xi-z_j)|^2\big)^{\frac{m+1}{2}}}\cdot\frac{1}{1+|x|^2}\cdot|\va(x+z_j)|d\vol_{\ig_{\R^m}}
\]
for $j\geq1$. By the H\"older's inequality and $\|\va\|\leq 1$, we have
\[
\aligned
I_j&\leq C \,\lm^{\frac{m+1}{2}}\Big( \int_{\R^m}\frac{1}{(\lm^2+|x-(\xi-z_j)|^2)^m}\cdot\frac1{(1+|x|^2 )^{\frac{2m}{m+1}}} d\vol_{\ig_{\R^m}}\Big)^{\frac{m+1}{2m}}.
\endaligned
\]
Recall that we have assumed $(\lm,\xi)\in\Om_H(z_i)$, then we claim that
\begin{\equ}\label{claim-jj}
J_j:=\int_{\R^m}\frac{1}{(\lm^2+|x-(\xi-z_j)|^2)^m}\cdot\frac1{(1+|x|^2 )^{\frac{2m}{m+1}}} d\vol_{\ig_{\R^m}}\leq \frac{C\lm^{-m}}{|\xi-z_j|^{\frac{4m}{m+1}}}
\end{\equ}
for some constant $C>0$ when $|\xi-z_j|\gg1$.

Assuming \eqref{claim-jj} for the moment, we soon have $I_j\leq C|\xi-z_j|^{-2}$, and hence
\[
\big\| \nabla\tilde\Ga(\psi_{\lm,\xi,\ga})-\nabla\Ga^{(i)}(\psi_{\lm,\xi,\ga}) \big\|\leq C\sum_{\substack{j\geq 1 \\ j\neq i}}\frac{|a_j|}{|\xi-z_j|^2}\leq CL^2\sum_{\substack{j\geq 1 \\ j\neq i}}\frac{|a_j|}{|z_j-z_i|^2}
\]
provided that $|z_j-z_i|\geq L$ with $L$ large enough (say $L\geq \diam\Om_H+1$). This proves $(2)$.

\medskip

Now it remains to prove \eqref{claim-jj}. Let us decompose the integral into two parts $J_j=J_{j,1}+J_{j,2}$, where
\[
J_{j,1}=\int_{|x|\leq \frac{|\xi-z_j|}2}\frac{1}{(\lm^2+|x-(\xi-z_j)|^2)^m}\cdot\frac1{(1+|x|^2 )^{\frac{2m}{m+1}}} d\vol_{\ig_{\R^m}}
\]
and
\[
J_{j,1}=\int_{|x|\geq \frac{|\xi-z_j|}2}\frac{1}{(\lm^2+|x-(\xi-z_j)|^2)^m}\cdot\frac1{(1+|x|^2 )^{\frac{2m}{m+1}}} d\vol_{\ig_{\R^m}}.
\]
Then, via elementary computations, we find
\[
\aligned
J_{j,1}&\leq  \frac1{\big( \lm^2+\big| \frac{\xi-z_j}2 \big|^2 \big)^m}\int_{|x|\leq \frac{|\xi-z_j|}2}\frac1{(1+|x|^2 )^{\frac{2m}{m+1}}} d\vol_{\ig_{\R^m}} \\[0.3em]
&\leq \frac{C}{|\xi-z_j|^{2m}}\int_0^{\frac{|\xi-z_j|}2}\frac{r^{m-1}dr}{(1+r^2)^{\frac{2m}{m+1}}}
\leq \begin{cases}
	C|\xi-z_j|^{-m-\frac{4m}{m+1}} & \text{if } m\geq4 \\[0.2em]
	C|\xi-z_j|^{-6}\ln|\xi-z_j| & \text{if } m=3 \\[0.2em]
	C|\xi-z_j|^{-4} & \text{if } m=2
\end{cases}
\endaligned
\]
and
\[
\aligned
J_{j,2}&\leq\frac1{\big( 1+\big|\frac{\xi-z_j}2\big|^2 \big)^{\frac{2m}{m+1}}}\int_{|x|\geq \frac{|\xi-z_j|}2}\frac{1}{(\lm^2+|x-(\xi-z_j)|^2)^m} d\vol_{\ig_{\R^m}}\\[0.3em]
&\leq \frac{C}{|\xi-z_j|^{\frac{4m}{m+1}}}\int_{\R^m}\frac1{(\lm^2+|x|^2)^m}d\vol_{\ig_{\R^m}}\leq \frac{C \lm^{-m}}{|\xi-z_j|^{\frac{4m}{m+1}}}
\endaligned
\]
which directly imply \eqref{claim-jj}. And the proof is hereby complete.
\end{proof}

Our next result intends to estimate the difference of the derivatives of the reduced functionals $\cj_\vr^{red}(z):=\cj_\vr(z+w_\vr(z))$ and $\cj_\vr^{(i),red}(z):=\cj_\vr^{(i)}(z+w_\vr^{i}(z))$ for $z=\psi_{\lm,\xi,\ga}\in\cm$ with $(\lm,\xi)\in\Om_H(z_i)$.

\begin{Prop}\label{Prop:diff-derivative}
Let $\{z_i\}\subset\R^m$ and $\{a_i\}\subset\R$ be chosen as above, then there exists a constant $C>0$ such that
\[
\|\nabla\tilde\Ga(z)-\nabla\Ga^{(i)}(z)\|\leq C R^{-2}
\]
for all $z=\psi_{\lm,\xi,\ga}\in\cm$ with $(\lm,\xi)\in\Om_H(z_i)$, some $i\geq1$, where $R\gg1$ is given in the definition of the sequence $\{z_i\}$. Furthermore, there holds
\[
\|\nabla \cj_\vr^{red}(z)-\nabla\cj_\vr^{(i),red}(z)\|\leq CR^{-2}|\vr|+o(\vr).
\] 
\end{Prop}
\begin{proof}
	Since
	\[
	\nabla \cj_\vr^{red}(z)=\vr\nabla\tilde\Ga(z)+o(\vr)
	\]
	and 
	\[
	\nabla\cj_\vr^{(i),red}(z)=\vr\nabla\Ga^{(i)}(z)+o(\vr),
	\]
	it follows from Lemma \ref{Lem:sum2} that
\begin{\equ}\label{diff1}
	\aligned
	\|\nabla \cj_\vr^{red}(z)-\nabla\cj_\vr^{(i),red}(z)\|&\leq |\vr|\|\nabla\tilde\Ga(z)-\nabla\Ga^{(i)}(z)\| + o(\vr) \\[0.3em]
	&\leq C |\vr| \sum_{\substack{j\geq 1 \\ j\neq i}}\frac{|a_j|}{|z_j-z_i|^2}+o(\vr).
	\endaligned
\end{\equ}
Next, we estimate $\sum_{j\neq i}|a_j||z_j-z_i|^{-2}$. Let us recall $z_j=j^\al z_0$, $|z_0|=R$ and $a_j=j^{-\bt}$ for $\al,\bt>1$ and $R\gg1$. Then we have
\[
\sum_{\substack{j\geq 1 \\ j\neq i}}\frac{|a_j|}{|z_j-z_i|^2}=R^{-2}\sum_{\substack{j\geq 1 \\ j\neq i}} \frac1{j^{\bt}}\cdot\frac1{|j^\al-i^\al|^2}.
\]
Notice that, for $i\leq j\leq i-1$, we have
\[
|j^\al-i^\al|=i^\al-j^\al\geq i^\al-(i-1)^\al\geq \al(i-1)^{\al-1}
\]
and similarly for $j\geq i+1$, we have
\[
|j^\al-i^\al|=j^\al-i^\al\geq (i+1)^\al-i^\al\geq\al i^{\al-1}.
\]
We thus have for $i\geq2$
\[
	\aligned
	\sum_{\substack{j\geq 1 \\ j\neq i}} \frac1{j^{\bt}}\cdot\frac1{|j^\al-i^\al|^2}&\leq \sum_{1\leq j<i}\frac1{j^\bt}\cdot\frac1{\al^2 (i-1)^{2\al-2}}+\sum_{j>i}\frac1{j^\bt}\cdot\frac1{\al^2 i^{2\al-2}} \\[0.3em]
	&\leq \frac{C(\bt)}{\al^2 (i-1)^{2\al-2}}\leq C(\bt)
	\endaligned
\]
where $C(\bt)=\sum_{j\geq1}\frac1{j^{\bt}}<+\infty$. The case $i=1$ is even much easier since we have
\[
\sum_{j=2}^\infty \frac1{j^{\bt}}\cdot\frac1{|j^\al-1|^2}< C(\bt).
\]
Hence, by \eqref{diff1}, we deduce
\[
\|\nabla \cj_\vr^{red}(z)-\nabla\cj_\vr^{(i),red}(z)\|\leq CC(\bt)R^{-2}|\vr|+o(\vr),
\]
which completes the proof.
\end{proof}

Now we are ready to prove Theorem \ref{main thm-H}.

\begin{proof}[Proof of Theorem \ref{main thm-H}]
Now, by Proposition \ref{prop:topology-degree}, \ref{Prop:diff-derivative} and the homotopy invariance of the topological degree, we can conclude that for $R\gg1$ sufficiently large (this only depends on the size of $\Om_H$) there exists $\vr_0>0$ (independent of $i$) such that
\[
\aligned
\deg(\nabla\tilde\Psi,\Om_H(z_i),0)&=\deg(\nabla\tilde\Ga,\Om_H(z_i)\times\{\ga\},0)
=\deg(\nabla \Ga^{(i)},\Om_H(z_i)\times\{\ga\},0)\\
&=\deg(\nabla \Psi^{(i)},\Om_H(z_i),0)\neq0
\endaligned
\]
for all $\vr\in(-\vr_0,\vr_0)$. Thus, there exists a critical point $(\lm_i,\xi_i)\in\Om_H(z_i)$ of the function $\tilde\Psi:\cg\to\R$, $(\lm,\xi)\mapsto\tilde\Psi(\lm,\xi)$. As we will see later in Appendix \ref{C2-smooth}, the function $\tilde K$ comes from a $C^2$-function on $S^m$ provided $\al,\,\bt>1$ satisfy $\bt>4\al+1$. Then it follows from \cite[Proposition 3.6]{Isboe-Xu21} that, by approximating $\tilde K$ if necessary, we may assume $\tilde\Psi$ is a Morse function. Therefore, by virtue of Theorem \ref{abstract result2}, we can choose $g_i(\ga)=(\lm_i(\ga),\xi_i(\ga))\in\Om_H(z_i)$ depending on $\ga\in\cn=S^{2^{[\frac m2]+1}-1}(\mbs_m)$ and $\ga_i\in \cn$ such that $\zeta_i:=(g_i(\ga_i),\ga_i)\in \cg\times\cn$ is a critical point of $\cj_\vr^{red}$. Then the critical point $\va_i=\zeta_i+w_\vr(\zeta_i)$ of $\cj_\vr$ are positive and concentrates at infinity. And, similar to the very last step in proving Theorem \ref{main thm}, the corresponding pull-back spinors $\{\psi_i\}$ concentrates at $p_0$ and $\|\psi_i\|_{L^\infty}\to+\infty$ as $i\to\infty$. Thus $\{\psi_i\}$ is a non-compact family of solutions to Eq. \eqref{SY-H-vr}. The estimates in Theorem \ref{main thm-H} $(2)$ simply come from some direct computations.
\end{proof}

%

\appendix
\section{Appendix}

\subsection{Proof of Lemma \ref{lem:inequalities}}\label{proof of lemma 4.5}

Since \eqref{inequ1} can be obtained by using the identification \eqref{Dirac identify} and Lemma \ref{basic expansions}, we start with \eqref{inequ2}.

For $\psi,\va\in\msd^{\frac12}(\R^m,\mbs(\R^m))\cap C^1(\R^m,\mbs(\R^m))$, there holds
\[
\inp{\nabla\cj_\vr(\tilde\psi)}{\tilde\va}=\frac12\real\int_{\R^m}(\tilde\psi,D_{\tilde\ig_\vr}\tilde\va)_{\tilde\ig_\vr}+(\tilde\va,D_{\tilde\ig_\vr}\tilde\psi)_{\tilde\ig_\vr}d\vol_{\tilde\ig_\vr}-\real\int_{\R^m}|\tilde\psi|_{\tilde\ig_\vr}^{2^*-2}(\tilde\psi,\tilde\va)_{\ig_\vr}d\vol_{\ig_\vr}.
\]
Using \eqref{Dirac identify}, we have
\begin{\equ}\label{AX1}
\aligned
&\real(\tilde\psi,D_{\tilde\ig_\vr}\tilde\va)_{\tilde\ig_\vr}+\real(\tilde\va,D_{\tilde\ig_\vr}\tilde\psi)_{\tilde\ig_\vr}\\[0.3em]
&\qquad =\real\big(\tilde\va,\widetilde{D_{\ig_{\R^m}}\psi}\big)_{\tilde\ig_\vr}+\real\big(\tilde\psi,\widetilde{D_{\ig_{\R^m}}\va}\big)_{\tilde\ig_\vr} +\real(\tilde\va,W\cdot_{\tilde\ig_\vr}\tilde\psi)_{\tilde\ig_\vr}\\[0.3em]
&\qquad \quad+\real(\tilde\psi,W\cdot_{\tilde\ig_\vr}\tilde\va)_{\tilde\ig_\vr} + \real(\tilde\va,X\cdot_{\tilde\ig_\vr}\tilde\psi)_{\tilde\ig_\vr}+ \real(\tilde\psi,X\cdot_{\tilde\ig_\vr}\tilde\va)_{\tilde\ig_\vr}\\[0.3em] 
&\qquad \quad + \sum_{i,j}(b_{ij}-\de_{ij})\real\big(\tilde\va,\tilde\pa_i\cdot_{\tilde\ig_\vr}\widetilde{\nabla_{\pa_j}\psi}\big)_{\tilde\ig_\vr} \\
&\qquad \quad + \sum_{i,j}(b_{ij}-\de_{ij})\real\big(\tilde\psi,\tilde\pa_i\cdot_{\tilde\ig_\vr}\widetilde{\nabla_{\pa_j}\va}\big)_{\tilde\ig_\vr}.
\endaligned
\end{\equ}
Notice that $X\in T\R^m$, we find
\[
\real(\tilde\va,X\cdot_{\tilde\ig_\vr}\tilde\psi)_{\tilde\ig_\vr}+ \real(\tilde\psi,X\cdot_{\tilde\ig_\vr}\tilde\va)_{\tilde\ig_\vr}=\real(\tilde\va,X\cdot_{\tilde\ig_\vr}\tilde\psi)_{\tilde\ig_\vr}- \real(X\cdot_{\tilde\ig_\vr}\tilde\psi,\tilde\va)_{\tilde\ig_\vr}=0.
\]
And using the explicit formula
\[
W = \frac14\sum_{\substack{i,j,k \\ i\neq j\neq k\neq i}}\sum_{\al,\bt} b_{i\al}(\pa_{\al}b_{j\bt})b_{\bt k}^{-1}\,\tilde\pa_i\cdot_{\tilde\ig_\vr} \tilde\pa_j\cdot_{\tilde\ig_\vr} \tilde\pa_k,
\] and Lemma \ref{basic expansions},
we can see that $W\equiv0$ in dimension $2$ and 	
\[
\left\{
\aligned
&b_{i\al}=\de_{i\al}-\frac\eps2\tilde\ih_{i\al}+\frac{3\,\eps^2}8\sum_l\tilde\ih_{il}\tilde\ih_{l\al}+o(\eps^2) , \\
&\pa_\al b_{j\bt}=-\frac\eps2\pa_\al\tilde\ih_{j\bt}+\frac{3\,\eps^2}8\sum_l\Big(
\pa_\al\tilde\ih_{jl}\tilde\ih_{l\bt}+\tilde\ih_{jl}\pa_\al\tilde\ih_{l\bt}\Big) + o(\eps^2), \\
&b_{\bt k}^{-1}=\de_{\bt k}+\frac\eps2\tilde\ih_{\bt k}-\frac{\eps^2}8
\sum_l\tilde\ih_{\bt l}\tilde\ih_{lk} + o(\eps^2),
\endaligned
\right.
\]
for dimension $m\geq3$. Hence
\[
\aligned
b_{i\al}(\pa_{\al}b_{j\bt})b_{\bt k}^{-1}&=-\frac\eps2\de_{\bt k}\de_{i\al}\pa_\al\tilde\ih_{j\bt}+\frac{\eps^2}4\pa_\al\tilde\ih_{j\bt}\big( \de_{\bt k}\tilde\ih_{i\al}-\de_{i\al}\tilde\ih_{\bt k} \big) \\
&\qquad +\frac{3\de_{\bt k}\de_{i\al}\eps^2}8\sum_l\Big(
\pa_\al\tilde\ih_{jl}\tilde\ih_{l\bt}+\tilde\ih_{jl}\pa_\al\tilde\ih_{l\bt}\Big) + o(\eps^2).
\endaligned
\]
Note that we have assumed $\tilde\ih_{ij}=0$ for $i\neq j$, we soon get
\[
b_{i\al}(\pa_{\al}b_{j\bt})b_{\bt k}^{-1}=o(\eps^2).
\]
Recalling that the map $\psi\mapsto\tilde\psi$ defined in \eqref{spinor identify} is fiberwisely isometric, we obtain
\[
\real\big(\tilde\va,\widetilde{D_{\ig_{\R^m}}\psi}\big)_{\tilde\ig_\vr}=\real\big(\va,D_{\ig_{\R^m}}\psi\big)_{\ig_{\R^m}}, \quad \real\big(\tilde\va,\tilde\pa_i\cdot_{\tilde\ig_\vr}\widetilde{\nabla_{\pa_j}\psi}\big)_{\tilde\ig_\vr}=\real\big(\va,\pa_i\cdot_{\ig_{\R^m}}\nabla_{\pa_j}\psi\big)_{\ig_{\R^m}}
\]
and
\[
\real(\tilde\va,W\cdot_{\tilde\ig_\vr}\tilde\psi)_{\tilde\ig_\vr}=\frac14
\sum_{\substack{i,j,k \\ i\neq j\neq k\neq i}}\Big(\sum_{\al,\bt} b_{i\al}(\pa_{\al}b_{j\bt})b_{\bt k}^{-1}\Big)\real(\pa_i\cdot_{\ig_{\R^m}} \pa_j\cdot_{\ig_{\R^m}} \pa_k\cdot_{\ig_{\R^m}}\psi,\va)_{\ig_{\R^m}}.
\]
And thus \eqref{AX1} can be expanded as
\[
\aligned
&\real(\tilde\psi,D_{\tilde\ig_\vr}\tilde\va)_{\tilde\ig_\vr}+\real(\tilde\va,D_{\tilde\ig_\vr}\tilde\psi)_{\tilde\ig_\vr}\\[0.3em]
&\qquad =\real(\va,D_{\ig_{\R^m}}\psi)_{\ig_{\R^m}}+\real(\psi,D_{\ig_{\R^m}}\va)_{\ig_{\R^m}} \\[0.3em]
&\qquad \quad -\frac\vr2\sum_i \tilde\ih_{ii}\big[ \real(\va,\pa_i\cdot_{\ig_{\R^m}}\nabla_{\pa_i}\psi)_{\ig_{\R^m}} +\real(\psi,\pa_i\cdot_{\ig_{\R^m}}\nabla_{\pa_i}\va)_{\ig_{\R^m}}\big]\\
&\qquad\quad +\frac{3\vr^2}8\sum_i\tilde\ih_{ii}^2\big[ \real(\va,\pa_i\cdot_{\ig_{\R^m}}\nabla_{\pa_i}\psi)_{\ig_{\R^m}} +\real(\psi,\pa_i\cdot_{\ig_{\R^m}}\nabla_{\pa_i}\va)_{\ig_{\R^m}}\big] \\
&\qquad \quad +o(\vr^2)|\va|_{\ig_{\R^m}}|\psi|_{\ig_{\R^m}}
\endaligned
\]
where the support of the last $o(\vr^2)$ term is contained in $\supp\tilde\ih$. Then, by \eqref{expansion det G}, the specific expression of $\cj_0$ and $\Ga$ in Lemma \ref{expansion J functional} with $\tilde\ih\in\ch(p)$ and the embedding $\msd^{\frac12}(\R^m,\mbs(\R^m))\hookrightarrow L^{2^*}(\R^m,\mbs(\R^m))$, we easily find
\[
\Big| \inp{\nabla\cj_\vr(\tilde\psi)}{\tilde\va}-\inp{\nabla\cj_0(\psi)}{\va}-\vr\inp{\nabla\Ga(\psi)}{\va} \Big|\leq O(\vr^2)\big(\|\psi\|\|\va\|+\|\psi\|^{2^*-1}\|\va\|\big)
\]
and \eqref{inequ2} is proved by using the fundamental fact that $\msd^{\frac12}(\R^m,\mbs(\R^m))\cap C^1(\R^m,\mbs(\R^m))$ is dense in $\msd^{\frac12}(\R^m,\mbs(\R^m))$.

The estimate \eqref{inequ3} can be obtained in a similar manner, in particular, we have $\|z\|$ is uniformly bounded for $z\in\cm$.

To see \eqref{inequ4}, let us remark that, 
\[
\aligned
\nabla^2\cj_\vr(\tilde\psi)[\tilde\phi,\tilde\va]&=\frac12\int_{\R^m}\real(\tilde\phi,D_{\tilde\ig_\vr}\tilde\va)_{\tilde\ig_\vr}+\real(\tilde\va,D_{\tilde\ig_\vr}\tilde\phi)_{\tilde\ig_\vr}\,d\vol_{\tilde\ig_\vr}
-\real\int_{\R^m}|\tilde\psi|_{\tilde\ig_\vr}^{2^*-2}(\tilde\phi,\tilde\va)_{\tilde\ig_\vr} \,d\vol_{\tilde\ig_\vr}\\
&\quad -(2^*-2)\int_{\R^m}|\tilde\psi|_{\tilde\ig_\vr}^{2^*-4}\real(\tilde\psi,\tilde\phi)_{\tilde\ig_\vr}\real(\tilde\psi,\tilde\va)_{\tilde\ig_\vr}\,d\vol_{\tilde\ig_\vr}
\endaligned
\]
for two given spinors $\va,\phi\in \msd^{\frac12}(\R^m,\mbs(\R^m))\cap C^1(\R^m,\mbs(\R^m))$. Then, by virtue of \eqref{AX1} and \eqref{expansion det G}, we deduce 
\[
\big| \nabla^2\cj_\vr(\tilde\psi)[\tilde\phi,\tilde\va]-\nabla^2\cj_0(\psi)[\phi,\va] \big|\leq O(\vr)\big( \|\phi\|\|\va\|+\|\psi\|^{2^*-2}\|\phi\|\|\va\| \big)
\]
which proves \eqref{inequ4} through the density of $\msd^{\frac12}(\R^m,\mbs(\R^m))\cap C^1(\R^m,\mbs(\R^m))$ in $\msd^{\frac12}(\R^m,\mbs(\R^m))$.

Next, let us turn to \eqref{inequ6}. Notice that there holds
\[
\aligned
&\inp{\nabla\cj_\vr(\tilde\psi+\tilde\va)}{\tilde\phi}-\inp{\nabla\cj_\vr(\tilde\psi)}{\tilde\phi}\\
&\qquad=\frac12\int_{\R^m}\real(\tilde\va,D_{\tilde\ig_\vr}\tilde\phi)_{\tilde\ig_\vr}+\real(\tilde\phi,D_{\tilde\ig_\vr}\tilde\va)_{\tilde\ig_\vr}\,d\vol_{\ig_\vr} \\
&\qquad \quad -\Big( \real\int_{\R^m}|\tilde\psi+\tilde\va|_{\tilde\ig_\vr}^{2^*-2}(\tilde\psi+\tilde\va,\tilde\phi)_{\tilde\ig_\vr}-|\tilde\psi|_{\tilde\ig_\vr}^{2^*-2}(\tilde\psi,\tilde\phi)_{\tilde\ig_\vr}\,d\vol_{\tilde\ig_\vr} \Big).
\endaligned
\]
This implies that
\begin{\equ}\label{AX2}
\aligned
&\big\|\nabla\cj_\vr(\tilde\psi+\tilde\va)-\nabla\cj_\vr(\tilde\psi) \big\|\\
&\qquad \leq O(1)\|\va\| + O(1)\Big( \int_{\R^m}\Big| |\psi+\va|_{\ig_{\R^m}}^{2^*-2}(\psi+\va)-|\psi|_{\ig_{\R^m}}^{2^*-2}\psi \Big|^{\frac{2m}{m+1}}d\vol_{\ig_{\R^m}} \Big)^{\frac{m+1}{2m}}.
\endaligned
\end{\equ}
Denoted by $f(s)=|\psi+s\va|_{\ig_{\R^m}}^{2^*-2}(\psi+s\va)$, we have
\[
 |\psi+\va|_{\ig_{\R^m}}^{2^*-2}(\psi+\va)-|\psi|_{\ig_{\R^m}}^{2^*-2}\psi=f(1)-f(0)=\int_0^1f'(s)ds
\]
and
\[
|f'(s)|\leq (2^*-1)|\psi+\va|_{\ig_{\R^m}}^{2^*-2}|\va|_{\ig_{\R^m}}.
\]
Using the H\"older inequality and Fubini Theorem, we have that
\[
\aligned
\int_{\R^m}|f(1)-f(0)|^{\frac{2m}{m+1}}d\vol_{\ig_{\R^m}}&\leq \int_{\R^m}\int_0^1|f'(s)|^{\frac{2m}{m+1}}ds\,d\vol_{\ig_{\R^m}} \\
&=\int_0^1\int_{\R^m}|f'(s)|^{\frac{2m}{m+1}}d\vol_{\ig_{\R^m}}ds \\
&\leq O(1)\int_0^1\int_{\R^m}|\psi+s\va|_{\ig_{\R^m}}^{\frac{2}{m-1}\cdot\frac{2m}{m+1}}|\va|_{\ig_{\R^m}}^{\frac{2m}{m+1}}d\vol_{\ig_{\R^m}}ds \\
&\leq O(1)\int_0^1|\psi+s\va|_{2^*}^{\frac{22^*}{m+1}}|\va|_{2^*}^{\frac{2m}{m+1}}ds \\
&\leq O(1)\|\va\|^{\frac{2m}{m+1}}\max_{s\in[0,1]}\|\psi+s\va\|^{\frac{22^*}{m+1}}
\endaligned
\]
So from \eqref{AX2} we deduce
\[
\big\|\nabla\cj_\vr(\tilde\psi+\tilde\va)-\nabla\cj_\vr(\tilde\psi) \big\|\leq O(1)\big( \|\va\|+ \|\va\|\max_{s\in[0,1]}\|\psi+\va\|^{\frac2{m-1}} \big)
\]
which suggests \eqref{inequ6}.

We point out that the estimates \eqref{inequ5} and \eqref{inequ7} can be obtained with similar procedures, and hence it remains to check \eqref{inequ8}.

Observe that, for two spinors $\phi_1,\phi_2\in\msd^{\frac12}(\R^m,\mbs(\R^m))$, we have
\begin{\equ}\label{AX3}
\aligned
&\nabla^2\cj_\vr(\tilde\psi+\tilde\va)[\tilde\phi_1,\tilde\phi_2]-\nabla^2\cj_\vr(\tilde\psi)[\tilde\phi_1,\tilde\phi_2] \\
&\qquad =-\real\int_{\R^m}|\tilde\psi+\tilde\va|_{\tilde\ig_\vr}^{2^*-2}(\tilde\phi_1,\tilde\phi_2)_{\tilde\ig_\vr}d\vol_{\ig_\vr}+\real\int_{\R^m}|\tilde\psi|_{\tilde\ig_\vr}^{2^*-2}(\tilde\phi_1,\tilde\phi_2)_{\tilde\ig_\vr}d\vol_{\ig_\vr}\\
&\qquad \quad -(2^*-2)\int_{\R^m}|\tilde\psi+\tilde\va|_{\tilde\ig_\vr}^{2^*-4}\real(\tilde\psi+\tilde\va,\tilde\phi_1)_{\tilde\ig_\vr}\real(\tilde\psi+\tilde\va,\tilde\phi_2)_{\tilde\ig_\vr}d\vol_{\ig_\vr} \\
&\qquad \quad +(2^*-2)\int_{\R^m}|\tilde\psi|_{\tilde\ig_\vr}^{2^*-4}\real(\tilde\psi,\tilde\phi_1)_{\tilde\ig_\vr}\real(\tilde\psi,\tilde\phi_2)_{\tilde\ig_\vr}d\vol_{\ig_\vr},
\endaligned
\end{\equ}
and
\begin{\equ}\label{AX4}
\aligned
&\Big|  \real\int_{\R^m}|\tilde\psi+\tilde\va|_{\tilde\ig_\vr}^{2^*-2}(\tilde\phi_1,\tilde\phi_2)_{\tilde\ig_\vr}d\vol_{\ig_\vr}-\real\int_{\R^m}|\tilde\psi|_{\tilde\ig_\vr}^{2^*-2}(\tilde\phi_1,\tilde\phi_2)_{\tilde\ig_\vr}d\vol_{\ig_\vr}\Big| \\[0.3em]
&\qquad \leq O(1)\int_{\R^m}\Big| |\tilde\psi+\tilde\va|_{\tilde\ig_\vr}^{2^*-2}-|\tilde\psi|_{\tilde\ig_\vr}^{2^*-2} \Big||\phi_1||\phi_2| d\vol_{\ig_{\R^m}}\\[0.3em]
&\qquad \leq \begin{cases}
\displaystyle 	O(1)\int_{\R^2}(|\psi|_{\ig_{\R^2}}|\va|_{\ig_{\R^2}}+|\va|_{\ig_{\R^2}}^2)|\phi_1|_{\ig_{\R^2}}|\phi_2|_{\ig_{\R^2}}d\vol_{\ig_{\R^2}} & \text{if } m=2 \\[1.3em]
\displaystyle O(1)\int_{\R^m} |\va|_{\ig_{\R^m}}^{2^*-2} |\phi_1|_{\ig_{\R^m}}|\phi_2|_{\ig_{\R^m}}d\vol_{\ig_{\R^m}} & \text{if } m\geq3
\end{cases}
\endaligned
\end{\equ}
where we have used the sub-additivity of the function $\psi\mapsto|\psi|^{2^*-2}$ for $2^*-2\in(0,1]$ (that is $m\geq3$). Thus, we only need to estimate the last two integrals in \eqref{AX3}. For this purpose, let us set
\[
\aligned
I_1&=\int_{\R^m}|\tilde\psi+\tilde\va|_{\tilde\ig_\vr}^{2^*-2}\frac{\real(\tilde\psi+\tilde\va,\tilde\phi_1)_{\tilde\ig_\vr}\real(\tilde\psi+\tilde\va,\tilde\phi_2)_{\tilde\ig_\vr}}{|\tilde\psi+\tilde\va|_{\ig_\vr}^2}d\vol_{\ig_\vr}\\[0.3em]
&\qquad -\int_{\R^m}|\tilde\psi|_{\tilde\ig_\vr}^{2^*-2}\frac{\real(\tilde\psi+\tilde\va,\tilde\phi_1)_{\tilde\ig_\vr}\real(\tilde\psi+\tilde\va,\tilde\phi_2)_{\tilde\ig_\vr}}{|\tilde\psi+\tilde\va|_{\ig_\vr}^2}d\vol_{\ig_\vr}
\endaligned
\]
and
\[
\aligned
I_2&=\int_{\R^m}|\tilde\psi|_{\tilde\ig_\vr}^{2^*-2}\frac{\real(\tilde\psi+\tilde\va,\tilde\phi_1)_{\tilde\ig_\vr}\real(\tilde\psi+\tilde\va,\tilde\phi_2)_{\tilde\ig_\vr}}{|\tilde\psi+\tilde\va|_{\ig_\vr}^2}d\vol_{\ig_\vr}\\[0.3em]
&\qquad -\int_{\R^m}|\tilde\psi|_{\tilde\ig_\vr}^{2^*-2}\frac{\real(\tilde\psi,\tilde\phi_1)_{\tilde\ig_\vr}\real(\tilde\psi,\tilde\phi_2)_{\tilde\ig_\vr}}{|\tilde\psi|_{\ig_\vr}^2}d\vol_{\ig_\vr}
\endaligned
\]
so that $I_1+I_2$ is nothing but the last two integrals in \eqref{AX3}. Clearly, $I_1$ can be estimated similar to \eqref{AX4}. And for $I_2$,
let us set $\Om=\big\{x\in\R^m:\, |\tilde\psi|_{\ig_\vr}/|\tilde\psi+\tilde\va|_{\ig_\vr}<2  \big\}$, then we can have the decomposition $I_2=I_2^{(1)}+I_2^{(2)}$ with $I_2^{(1)}$ and $I_2^{(2)}$ being the integration on $\Om$ and $\R^m\setminus\Om$, respectively. Notice that, on $\R^m\setminus\Om$, we have $|\tilde\psi|_{\ig_\vr}\leq 2|\tilde\va|_{\ig_\vr}$. Hence, there holds
\[
|I_2^{(2)}|\leq O(1)\int_{\R^m\setminus\Om}|\tilde\psi|_{\ig_\vr}^{2^*-2}|\tilde\phi_1|_{\ig_\vr}|\tilde\phi_2|_{\ig_\vr}d\vol_{\ig_\vr}\leq O(1)\int_{\R^m}|\va|_{\ig_{\R^m}}^{2^*-2}|\phi_1|_{\ig_{\R^m}}|\phi_2|_{\ig_{\R^m}}d\vol_{\ig_{\R^m}}.
\]
Meanwhile, by using the fact
\[
\Big| \frac{\tilde\psi+\tilde\va  }{|\tilde\psi+\tilde\va|_{\ig_\vr}}-\frac{\tilde\psi}{|\tilde\psi|_{\ig_\vr}} \Big|_{\ig_\vr}=\Big| \frac{\tilde\psi|\tilde\psi|_{\ig_\vr}+\tilde\va|\tilde\psi|_{\ig_\vr}-\tilde\psi|\tilde\psi+\tilde\va|_{\ig_\vr}}{|\tilde\psi+\tilde\va|_{\ig_\vr}|\tilde\psi|_{\ig_\vr}} \Big|_{\ig_\vr}\leq \frac{2|\tilde\va|_{\ig_\vr}}{|\tilde\psi+\tilde\va|_{\ig_\vr}}
\]
and
\[
\aligned
I_2^{(1)}&=\int_{\Om}|\tilde\psi|_{\tilde\ig_\vr}^{2^*-2}\frac{\real(\tilde\psi+\tilde\va,\tilde\phi_1)_{\tilde\ig_\vr}\real(\tilde\psi+\tilde\va,\tilde\phi_2)_{\tilde\ig_\vr}}{|\tilde\psi+\tilde\va|_{\ig_\vr}^2}d\vol_{\ig_\vr} \\[0.3em]
&\qquad -\int_{\Om}|\tilde\psi|_{\tilde\ig_\vr}^{2^*-2}\frac{\real(\tilde\psi,\tilde\phi_1)_{\tilde\ig_\vr}\real(\tilde\psi+\tilde\va,\tilde\phi_2)_{\tilde\ig_\vr}}{|\tilde\psi|_{\ig_\vr}|\tilde\psi+\tilde\va|_{\ig_\vr}}d\vol_{\ig_\vr}
\\[0.3em]
&\qquad 
+\int_{\Om}|\tilde\psi|_{\tilde\ig_\vr}^{2^*-2}\frac{\real(\tilde\psi,\tilde\phi_1)_{\tilde\ig_\vr}\real(\tilde\psi+\tilde\va,\tilde\phi_2)_{\tilde\ig_\vr}}{|\tilde\psi|_{\ig_\vr}|\tilde\psi+\tilde\va|_{\ig_\vr}}d\vol_{\ig_\vr}
\\[0.3em]
&\qquad -\int_{\Om}|\tilde\psi|_{\tilde\ig_\vr}^{2^*-2}\frac{\real(\tilde\psi,\tilde\phi_1)_{\tilde\ig_\vr}\real(\tilde\psi,\tilde\phi_2)_{\tilde\ig_\vr}}{|\tilde\psi|_{\ig_\vr}^2}d\vol_{\ig_\vr}
\endaligned
\]
we deduce 
\begin{\equ}\label{AX5}
\aligned
|I_2^{(1)}|&\leq 2\int_{\Om}|\tilde\psi|_{\ig_\vr}^{2^*-2}\Big| \frac{\tilde\psi+\tilde\va  }{|\tilde\psi+\tilde\va|_{\ig_\vr}}-\frac{\tilde\psi}{|\tilde\psi|_{\ig_\vr}} \Big|_{\ig_\vr}|\tilde\phi_1|_{\ig_\vr}|\tilde\phi_2|_{\ig_\vr}d\vol_{\ig_\vr} \\[0.3em]
&\leq O(1)\int_{\Om}\frac{|\tilde\psi|_{\ig_\vr}^{2^*-2}|\tilde\va|_{\ig_\vr}|\tilde\phi_1|_{\ig_\vr}|\tilde\phi_2|_{\ig_\vr}}{|\tilde\psi+\tilde\va|_{\ig_\vr}}d\vol_{\ig_\vr} \\[0.3em]
&\leq O(1)\int_{\R^m}|\psi|_{\ig_{\R^m}}^{2^*-3}|\va|_{\ig_{\R^m}}|\phi_1|_{\ig_{\R^m}}|\phi_2|_{\ig_{\R^m}}d\vol_{\ig_{\R^m}}.
\endaligned
\end{\equ}
And thus, we obtain 
\begin{\equ}\label{AX6}
|I_2|\leq \begin{cases}
\displaystyle O(1)\int_{\R^2}(|\psi|_{\ig_{\R^2}}|\va|_{\ig_{\R^2}}+|\va|_{\ig_{\R^2}}^2)|\phi_1|_{\ig_{\R^2}}|\phi_2|_{\ig_{\R^2}}d\vol_{\ig_{\R^2}} & \text{if } m=2 \\[1.3em]
\displaystyle O(1)\int_{\R^3}|\va|_{\ig_{\R^3}}|\phi_1|_{\ig_{\R^3}}|\phi_2|_{\ig_{\R^3}}d\vol_{\ig_{\R^2}} & \text{if } m=3 
\end{cases}
\end{\equ}
Notice that $2^*=\frac{2m}{m-1}<3$ for $m\geq4$, we need to divide $\Om$ into two parts, i.e. $\Om=\Om_1\cup\Om_2$ with $\Om_1:=\big\{x\in\Om:\, |\psi|_{\ig_{\R^m}}>|\va|_{\ig_{\R^m}}\big\}$ and $\Om_2:=\big\{x\in\Om:\, |\psi|_{\ig_{\R^m}}\leq|\va|_{\ig_{\R^m}}\big\}$. Then, from the first and second lines in \eqref{AX5}, we obtain
\[
|I_2^{(1)}|\leq O(1)\int_{\Om_1}|\va|_{\ig_{\R^m}}^{2^*-2}|\phi_1|_{\ig_{\R^m}}|\phi_2|_{\ig_{\R^m}}d\vol_{\ig_{\R^m}}+ O(1)\int_{\Om_2}|\psi|_{\ig_{\R^m}}|\phi_1|_{\ig_{\R^m}}|\phi_2|_{\ig_{\R^m}}d\vol_{\ig_{\R^m}}.
\]
Hence we have
\begin{\equ}\label{AX7}
|I_2|\leq O(1)	\int_{\R^m}\big(|\va|_{\ig_{\R^m}}^{2^*-2}+|\va|\big)|\phi_1|_{\ig_{\R^m}}|\phi_2|_{\ig_{\R^m}}d\vol_{\ig_{\R^m}} \quad \text{for } m\geq4.
\end{\equ}
Now, combining \eqref{AX3}-\eqref{AX7}, we find that
\[
\big\|\nabla^2\cj_\vr(\tilde\psi+\tilde\va)-\nabla^2\cj_\vr(\tilde\psi)\big\|
\leq \begin{cases}
O(1)\big(\|\psi\|\|\va\|+\|\va\|^2 \big) & \text{if } m=2 \\[0.5em]
O(1)\big( \|\va\|^{2^*-2}+\|\va\| \big) & \text{if } m\geq3
\end{cases}
\]
which proves \eqref{inequ8}. And the proof is hereby completed.

\subsection{The global $C^2$ smoothness of the pull-back function $\tilde K\circ\pi_{p_0}$ on $S^m$}\label{C2-smooth}

Here we show that $\tilde K$ comes from a $C^2$-function on $S^m$ when $\al,\bt>0$ satisfy $\bt>4\al+1$. And this will complete the proof of Theorem \ref{main thm-H}.

\medskip

Clearly, $\tilde K$ is $C^2$ on $\R^m$, because the series defining $\tilde K$ converges uniformly on $\R^m$ up to the second derivatives. To prove the differentiability at infinity (which correspond to the north pole of $S^m$), we need to show that $y\mapsto\tilde K(y/|y|^2)$ is twice continuously differentiable near $y=0$. Without loss of generality, we assume $|y|<1$ in the following context. And, by Lemma \ref{Lem:sum2}, we see that $\tilde K(y/|y|^2)$ converges uniformly in $y$. In particular, we have $\tilde K(y/|y|^2)\to0$ as $y\to0$. 

To see the convergence of the derivatives, for the function $K$ as before, we define $\hat K(y)=K\big(\frac{y}{|y|^2}\big)$. Then, an elementary computation shows that derivatives of $\hat K$ can be estimated as
\[
|\nabla\hat K(y)|\leq C\,\Big| \nabla K\Big( \frac{y}{|y|^2} \Big) \Big||y|^{-2}
\]
and
\[
|\nabla^2\hat K(y)|\leq C \bigg( \Big| \nabla^2 K\Big( \frac{y}{|y|^2} \Big) \Big||y|^{-4}+ \Big| \nabla K\Big( \frac{y}{|y|^2} \Big) \Big||y|^{-3} \bigg).
\]

Recall \eqref{K-inequ-1}, we notice that
\begin{\equ}\label{K-D1}
	\Big| \nabla K\Big( \frac{y}{|y|^2}-z_i \Big) \Big| \leq \frac{C_0}{1+\big| \frac{y}{|y|^2}-z_i \big|^2}\leq\frac{C_0}{1+\big| \frac{1}{|y|}-i^\al R \big|^2}
\end{\equ}
where in the last inequality we used $z_i=i^\al z_0$ with $|z_0|=R$ and the triangle inequality $\big| \frac{y}{|y|^2}-z_i \big|\geq\big| \frac{1}{|y|}-i^\al R \big|$. Then, by \eqref{K-D1} and our choice $a_i=i^{-\bt}$, we have for $i\geq N$ ($N\in\N$ is arbitrarily large)
\begin{\equ}\label{DK1}
\sum_{i\geq N}|a_i|\Big| \nabla\Big( K\Big( \frac{y}{|y|^2}-z_i \Big) \Big) \Big|\leq C|y|^{-2}\sum_{i\geq N} i^{-\bt}\cdot\frac{1}{1+\big| \frac{1}{|y|}-i^\al R \big|^2}.
\end{\equ}
Let us set
\[
S(N,y)=\sum_{i\geq N} i^{-\bt}\cdot\frac{1}{1+\big| \frac{1}{|y|}-i^\al R \big|^2}.
\]
To obtain an uniform estimate of $S(N,y)$ for $|y|\leq 1$, we decompose the sum into two pieces: $(i)$ $|y|\leq \frac1{2N^\al R}$ and $(ii)$ $|y|> \frac1{2N^\al R}$. For $(i)$, we have 
\begin{\equ}\label{S1}
\aligned
S(N,y)&=\sum_{N\leq i \leq (2|y|R)^{-\frac1\al}} i^{-\bt}\cdot\frac{1}{1+\big| \frac{1}{|y|}-i^\al R \big|^2}+\sum_{i > (2|y|R)^{-\frac1\al}} i^{-\bt}\cdot\frac{1}{1+\big| \frac{1}{|y|}-i^\al R \big|^2} \\[0.3em]
&\leq C \sum_{i\geq N} \frac{|y|^2}{i^\bt}+\sum_{i > (2|y|R)^{-\frac1\al}} \frac1{i^\bt}
\leq C |y|^2 N^{1-\bt}+ C |y|^{\frac{\bt-1}{\al}}.
\endaligned
\end{\equ}
And, for $(ii)$, we have
\begin{\equ}\label{S2}
	S(N,y)\leq \sum_{i\geq N} \frac1{i^\bt}\leq C N^{1-\bt}
\end{\equ}
Therefore, by additionally requiring $\bt>2\al+1$, we can deduce for the case $(i)$
\[
|y|^{-2}S(N,y)\leq CN^{1-\bt}+C|y|^{\frac{\bt-1}{\al}-2}\leq
CN^{1-\bt}+CN^{1-\bt+2\al}\leq CN^{1-\bt+2\al},
\]
and for the case $(ii)$
\[
|y|^{-2}S(N,y)\leq C |y|^{-2} N^{1-\bt}\leq C N^{1-\bt+2\al},
\]
where $N$ is considered arbitrarily large. Thus, the estimates in \eqref{S1} and \eqref{S2} imply that
\[
\sup_{|y|\leq 1}|y|^{-2}S(N,y)=O(N^{1-\bt+2\al}) \quad \text{as } N\to+\infty.
\]
And hence, by \eqref{DK1}, the series defining $\nabla\big(\tilde K(y/|y|^2)\big)$ converges uniformly on $|y|\leq1$. This suggests that $\tilde K(y/|y|^2)$ can be extended to $y=0$ in the class of $C^1$.

The second derivatives can be estimated in a similar manner. At this stage, instead of \eqref{K-inequ-1}, we need the following improved estimates 
\begin{\equ}\label{K-inequ-1-improved}
|\nabla K(x)|\leq C_0(1+|x|^2)^{-\frac32} \quad \text{and} \quad |\nabla^2K(x)|\leq C_0(1+|x|^2)^{-2}
\end{\equ}
by the choice of $p_0$. In fact, we have
\[
\nabla K(x)=\nabla H(\pi_{p_0}^{-1}(x))[\nabla \pi_{p_0}^{-1}(x)],
\]
\[
\nabla^2K(x)=\nabla^2 H(\pi_{p_0}^{-1}(x))[\nabla \pi_{p_0}^{-1}(x),\nabla \pi_{p_0}^{-1}(x)]+\nabla H(\pi_{p_0}^{-1})[\nabla^2\pi_{p_0}^{-1}(x)]
\]
and, as in the proof of Lemma \ref{K-inequ-2},
\[
|\nabla H(\pi_{p_0}^{-1}(x))|=|\nabla H(\pi_{p_0}^{-1}(x))-\nabla H(p_0)|\leq C \max_{S^m}|\nabla^2 H|\cdot |\pi_{p_0}^{-1}(x)-p_0|\leq C (1+|x|^2)^{-\frac12}.
\]
These, together with the facts
\[
|\nabla\pi_{p_0}^{-1}(x)|\leq C(1+|x|^2)^{-1} \quad \text{and} \quad |\nabla^2\pi_{p_0}^{-1}(x)|\leq C(1+|x|^2)^{-\frac32},
\]
we obtain \eqref{K-inequ-1-improved}. 

Now, by using the estimate
\[
\aligned
&\Big| \nabla^2 K\Big( \frac{y}{|y|^2}-z_i \Big) \Big||y|^{-4}+ \Big| \nabla K\Big( \frac{y}{|y|^2} -z_i\Big) \Big||y|^{-3} \\[0.3em]
&\qquad \leq \frac{C}{|y|^4}\cdot\frac1{\big( 1+\big|\frac1{|y|}-i^\al R\big|^2 \big)^{2}}+ \frac{C}{|y|^3}\cdot\frac1{\big( 1+\big|\frac1{|y|}-i^\al R\big|^2 \big)^{\frac32}},
\endaligned
\]
we find
\begin{\equ}\label{}
	\aligned
\sum_{i\geq N} |a_i|\Big| \nabla^2\Big( K\Big(
\frac{y}{|y|^2}-z_i\Big) \Big) \Big| & \leq C|y|^{-4}\sum_{i\geq N} \frac1{i^\bt}\cdot\frac1{\big( 1+\big|\frac1{|y|}-i^\al R\big|^2 \big)^{2}} \\
&\qquad + C|y|^{-3}\sum_{i\geq N}\frac1{i^\bt}\cdot\frac1{\big( 1+\big|\frac1{|y|}-i^\al R\big|^2 \big)^{\frac32}},
\endaligned
\end{\equ}
where $N$ is arbitrarily large as before. Let us set
\[
\tilde S_1(N,y)=\sum_{i\geq N} \frac1{i^\bt}\cdot\frac1{\big( 1+\big|\frac1{|y|}-i^\al R\big|^2 \big)^{2}} \quad \text{and} \quad \tilde S_2(N,y)=\sum_{i\geq N}\frac1{i^\bt}\cdot\frac1{\big( 1+\big|\frac1{|y|}-i^\al R\big|^2 \big)^{\frac32}}
\]
then, by performing the same arguments in \eqref{S1} and \eqref{S2}, we soon get
\[
|y|^{-4}\tilde S_1(N,y) = O(N^{1-\bt+4\al})\quad \text{and} \quad |y|^{-3}\tilde S_2(N,y)=O(N^{1-\bt+4\al}) \quad \text{as } N\to+\infty
\]
provided that $\bt>4\al+1$. Thus, in this case, the series defining $\nabla^2(\tilde K(y/|y|^2))$ converges uniformly on $|y|\leq 1$. This proves that $\tilde K(y/|y|^2)$ can be extended to $y=0$ in the class of $C^2$, when $\bt>4\al+1$.


\begin{thebibliography}{SK}
		
		\normalsize
		\baselineskip=16pt
		
	\bibitem{Adams}R. Adams,
	Sobolev Spaces,
	Academic Press, New York 1975.
	
	\bibitem{AB}A. Ambrosetti, M. Badiale,
	Homoclinics: Poincar\'e-Melnikov type results via a variational approach,
	Ann. Inst. H. Poincar\'e Anal. Nonlin\'eaire 15 (1998), 233-252.
	
	\bibitem{AB98}A. Ambrisetti, M. Badiale,
	Variational perturbative methods and bifurcation of bound states from the essential spectrum,
	Proc. Roy. Soc. Edinburgh Sect. A 128 (1998), 1131-1161.
	
	\bibitem{AM99} A. Ambrosetti, A. Malchiodi, 
	A multiplicity result for the Yamabe problem on $S^n$,
	J. Funct. Anal. 168 (1999), 529-561.
	
	\bibitem{AM}A. Ambrosetti, A. Malchiodi,
	Perturbation Methods and Semilinear Elliptic Problems on $\R^n$,
	Progress in Math. 240, Birkh\"auser, Basel-Boston-Berlin (2006).
	
	
	\bibitem{Ammann}B. Ammann,
	A variational problem in conformal spin geometry,
	Habilitationsschrift, Universit\"at Hamburg 2003.
	
	\bibitem{Ammann2003} B. Ammann,
	A spin-conformal lower bound of the first positive Dirac eigenvalue,
	Differ. Geom. Appl. 18:1 (2003), 21-32.
	
	\bibitem{Ammann2009} B. Ammann,
	The smallest Dirac eigenvalue in a spin-conformal class and cmc immersions, 
	Commun. Anal. Geom. 17:3 (2009), 429-479.
	
	\bibitem{ADHH}B. Ammann, M. Dahl, A. Hermann, E. Humbert,
	Mass endomorphism, surgery and perturbations,
	Ann. Inst. Fourier. 64 (2014), no. 2, 467-487.
	
	\bibitem{AGHM}B. Ammann, J.-F. Grossjean, E. Humbert, B. Morel,
	A spinorial analogue of Aubin's inequality,
	Math. Z. 260 (2008), 127-151.
	
	\bibitem{AHA}B. Ammann, E. Humbert, M. Ould. Ahmedou,
	An obstruction for the mean curvature of a conformal immersion $S^n\to\R^{n+1}$,
	Proc. Amer. Math. Soc. 135 (2007), no. 2, 489-493.
	
	\bibitem{AHM}B. Ammann, E. Humbert, B. Morel,
	Mass endomorphism and spinorial Yamabe type problems on conformally flat manifolds,
	Comm. Anal. Geom. 14:1 (2006), 163-182.
	
	\bibitem{Aubin}T. Aubin, 
	\'Equations diff\'erentielles non lin\'eaires et probl\'eme de Yamabe concernant la courbure scalaire,
	J. Math. Pures Appl. 55 (1976), 269-296.
	

\bibitem{BC91}A. Bahri, J. M. Coron, 
The scalar curvature problem on the standard three-dimensional sphere,
J. Funct. Anal. 95 (1991), 106-172.
	
	\bibitem{Ba-Xu-JFA21}T. Bartsch, T. Xu,
	A spinorial analogue of the Brezis-Nirenberg theorem involving the critical Sobolev exponent,
	J. Funct. Anal. 280 (2021), 108991.
	
	\bibitem{BM} M. Berti, A. Malchiodi,
	Non-compactness and Multiplicity Results for the Yamabe Problem on $S^n$,
	J. Funct. Anal. 180 (2001), 210-241.
	
	\bibitem{BF20} W. Borrelli, R.L. Frank,
	Sharp decay estimates for critical Dirac equations,
	Trans. Amer. Math. Soc. 373 (2020), 2045-2070.
	
	\bibitem{BMW21} W. Borrelli, A. Malchiodi, R. Wu,
	Ground state Dirac bubbles and Killing spinors,
	Commun. Math. Phys. 383 (2021), 1151-1180.
	
	\bibitem{BG}J.-P. Bourguignon, P. Gauduchon,
	Spineurs, op\'erateurs de Dirac et variations de m\'etriques, Commun. Math. Phys. 144 (1992), 581-599.
	
	\bibitem{Brendle}S. Brendle,
	Blow-up phenomena for the Yamabe equation,
	J. Amer. Math. Soc. 21 (2008), 951-979.
	
	\bibitem{BrMa} S. Brendle, F.C. Marques,
	Blow-up phenomena for the Yamabe equation II,
	J. Differential Geom. 81 (2009), 225-250.
	
	\bibitem{Chang}K.-C. Chang,
	Infinite Dimensional Morse Theory and Multiple Solution Problems,
	Birkh\"auser Boston (1993).
	
	 \bibitem{CGY93}S.Y.A. Chang, M. Gursky, P. Yang, 
	The scalar curvature equation on 2- and 3-spheres, 
	Calc. Var. 1 (1993), 205-229.
	
	 \bibitem{CY91}S.Y.A. Chang, P. Yang,
	 A perturbation result in prescribing scalar curvature on $S^n$,
	 Duke Math. J. 64 (1991), 27-69.	 
	
	\bibitem{CJW}Q. Chen, J. Jost, G. Wang,
	Nonlinear Dirac equations on Riemann surfaces,
	Ann. Global Anal. Geom. 33 (2008), 253-270.
	
	\bibitem{Fredrich98} T. Friedrich, 
	On the spinor representation of surfaces in Euclidean 3-space, 
	J. Geom. Phys. 28 (1998), 143-157.
		
	\bibitem{Friedrich00}T. Friedrich,
	Dirac Operators in Riemannian Geometry,
	Graduate Studies in Mathematics 25 (2000).	
	
	\bibitem{Ginoux} N. Ginoux,
	The Dirac Spectrum,
	Lecture Notes in Mathematics, vol. 1976. Springer, Berlin 2009.
	
	\bibitem{GB98} K. Grosse-Brauckmann, 
	Complete embedded constant mean curvature surfaces,
	Habilitationsschrift, Universit\"at Bonn, 1998.
	
	\bibitem{Hermann10}A. Hermann, 
	Generic metrics and the mass endomorphism on spin three-manifolds,
	Ann. Global Anal. Geom. 37 (2010), 163-171.
	
	\bibitem{Hopf} H. Hopf, Differential geometry in the large. In: Lecture Notes in Mathematics, vol. 1000.
	Springer, Berlin. Notes taken by Peter Lax and John Gray. With a preface by S.S. Chern
	(1983). 
	
	\bibitem{Isobe11}T. Isobe, 
	Nonlinear Dirac equations with critical nonlinearities on compact spin manifolds,
	J. Funct. Anal. 260 (2011), 253-307.
	
	\bibitem{Isobe13}T. Isobe,
	A perturbation method for spinorial Yamabe type equations on $S^m$ and its applications,
	Math. Ann. 355 (2013), 1255-1299.
	
	\bibitem{Isobe15}T. Isobe,
	Spinorial Yamabe type equations on $S^3$ via Conley index,
	Adv. Nonlinear. Stud. 15(2015), 39-60.
	
	\bibitem{Isboe-Xu21}T. Isobe, Tian Xu, 
	Solutions of Spinorial Yamabe-type Problems on $S^m$: Perturbations and Applications, Trans. Amer. Math. Soc. (2023), accepted for publication, http://arxiv.org/abs/2304.02807
	
	\bibitem{Kenmotsu}K. Kenmotsu,
	Weierstrass formula for surfaces of prescribed mean curvature,
	Math. Ann. 245:2 (1979), 89-99.
	
	\bibitem{KMS}M. Khuri, F.C. Marques, R. Schoen, 
	A compactness theorem for the Yamabe
	problem, 
	J. Differential Geom. 81 (2009), 143-196.
	
	\bibitem{Konopelchenko}B. G. Konopelchenko,
	Induced surfaces and their integrable dynamics,
	Stud. Appl. Math. 96:1 (1996), 9-51.
	
	\bibitem{KS96} R. Kusner, N. Schmitt, 
	The spinor representation of surfaces in space, preprint, http://www.arxiv.org/abs/dg-ga/9610005 (1996).
		
		
	\bibitem{LM}H. B. Lawson, M. L. Michelson,
	Spin Geometry,
	Princeton University Press (1989).
	
	\bibitem{LeeParker}J.M. Lee, T. Parker,
	The Yamabe problem, 
	Bull. Amer. Math. Soc. 17 (1987), 37-91.
	
	\bibitem{LY82}P. Li, S.T. Yau, 
	A new conformal invariant and its applications to the Willmore conjecture
	and the first eigenvalue on compact surfaces, 
	Invent. Math. 69 (1982), 269-291.
	
	\bibitem{Matsutani}S. Matsutani,
	Immersion anomaly of Dirac operator on surface in $\R^3$,
	Rev. Math. Phys. 11:2 (1999), 171-186.
	
	\bibitem{MW}J. Mawhin, M.Willem,
	Critical Point Theory and Hamiltonian Systems,
	Applied Mathematical Sciences 74 (1989), Springer-Verlag.
	
	\bibitem{Raulot}S. Raulot,
	A Sobolev-like inequality for the Dirac operator,
	J. Funct. Anal. 26 (2009), 1588-1617.
	
	\bibitem{Schoen}R. Schoen,
	Conformal deformation of a Riemannian metric to constant scalar curvature, 
	J. Differ. Geom. 20 (1984), 479-495.
	
	\bibitem{Schoen1991}R. Schoen, 
	On the number of constant scalar curvature metrics in a conformal class, in: Differential Geometry, in: Pitman Monogr. Surveys Pure Appl. Math., vol. 52, Longman Sci. Tech, Harlow, 1991, pp. 311-320.
	
	\bibitem{SX2020} Y. Sire, T. Xu,
	A variational analysis of the spinorial Yamabe equation on product manifolds,
	Ann. Scuola Norm. Sup. Pisa Cl. Sci. 24:5 (2023), 205-248.
	
	\bibitem{SX2021} Y. Sire, T. Xu,
	On the B\"ar-Hijazi-Lott invariant for the Dirac operator and a spinorial proof of the Yamabe problem, arXiv:2112.03640.
	
	\bibitem{Taimanov97}I. A. Taimanov,
	Modified Novikov-Veselov equation and differential
	geometry of surfaces,
	Amer. Math. Soc. Transl. (2) 179 (1997), 133-151;
	http://arxiv.org/dg-ga/9511005.
	
	\bibitem{Taimanov98}I. A. Taimanov,
	The Weierstrass representation of closed surfaces in $\R^3$, Funktsional. Anal. i Prilozhen. 32:4 (1998), 49-62; English transl., Funct. Anal. Appl. 32 (1998), 258-267.
	
	\bibitem{Taimanov99}I. A. Taimanov,
	The Weierstrass representation of spheres in $\R^3$, the Willmore numbers, and soliton spheres,
	Trudy Mat. Inst. Steklov. 225 (1999), 339-361; English transl.,
	Proc. Steklov Inst. Math. 225 (1999), 225-243.
	
	\bibitem{Trudinger}N. Trudinger, 
	Remarks concerning the conformal deformation of Riemannian structures on compact manifolds, 
	Ann. Scuola Norm. Sup. Pisa Cl. Sci. 22:4 (1968), 265-274.
	
	\bibitem{Xu-CAG} T. Xu, Conformal embeddings of $S^2\to\R^3$ with prescribed mean curvature: A variational approach, Commun. Anal. Geom. (2022), accepted for publication, https://arxiv.org/abs/1810.03874.
	
	\bibitem{Yamabe}H. Yamabe,
	On a deformation of Riemannian structures on compact manifolds, 
	Osaka Math. J. 12 (1960), 21-37.
		
	\end{thebibliography}
\end{document}